\documentclass[11pt]{amsart}
\usepackage[margin=1in]{geometry}
\usepackage{amsmath,amssymb,amsthm,mathtools,bm}
\usepackage{enumitem}
\usepackage{booktabs}
\usepackage{hyperref}
\usepackage[nameinlink,capitalize]{cleveref}
\usepackage[dvipsnames]{xcolor}
\usepackage{cancel}
\usepackage{mathrsfs}
\usepackage{algpseudocode,algorithm,algorithmicx}
\usepackage{tikz}
\usepackage{graphicx}
\usepackage{subcaption}
\usepackage[T1]{fontenc}
\usepackage[scaled=0.95]{helvet}
\usetikzlibrary{arrows.meta,calc,positioning,patterns,decorations.pathreplacing}
\usepackage{pgfplots}
\pgfplotsset{compat=1.18}

\hypersetup{colorlinks=true,linkcolor=blue,citecolor=blue,urlcolor=blue}

\newtheorem{theorem}{Theorem}[section]
\newtheorem{corollary}[theorem]{Corollary}
\newtheorem{proposition}[theorem]{Proposition}
\newtheorem{lemma}[theorem]{Lemma}
\newtheorem{definition}[theorem]{Definition}
\newtheorem{remark}[theorem]{Remark}

\newcommand{\commentout}[1]{}

\newcommand{\R}{\mathbb R}
\newcommand{\C}{\mathbb C}

\newcommand{\CC}{\mathbb C}

\newcommand{\Id}{I}

\newcommand{\bmC}{\mathbb{C}}
\newcommand{\bmS}{\mathbb{S}}
\newcommand{\cH}{\mathcal{H}}
\newcommand{\cU}{\mathcal{U}}

\newcommand{\dd}{\,d}

\newcommand{\Sph}{\mathbb S^2}

\newcommand{\diag}{\operatorname{diag}}
\newcommand{\sinc}{\operatorname{sinc}}
\newcommand{\Proj}{P}
\newcommand{\epssub}{\varepsilon_{\rm sub}}
\newcommand{\norm}[1]{\left\|#1\right\|_2}
\newcommand{\abs}[1]{\left|#1\right|}
\newcommand{\inner}[2]{\left\langle #1,#2\right\rangle}

\newcommand{\bd}{{\bm d}}
\newcommand{\bg}{{\bm g}}
\newcommand{\bu}{{\bm u}}

\newcommand{\bx}{{\bm x}}
\newcommand{\by}{{\bm y}}
\newcommand{\bz}{{\bm z}}
\newcommand{\bh}{{\bm h}}
\newcommand{\bv}{{\bm v}}
\newcommand{\bw}{{\bm w}}

\newcommand{\bom}{\bm\omega}

\newcommand{\eps}{\varepsilon}

\newcommand{\ip}[2]{\left\langle #1,#2\right\rangle}
\newcommand{\mesh}{{\rm mesh}}
\newcommand{\ii}{\mathrm i}

\newcommand{\beq}{\begin{eqnarray}}
\newcommand{\eeq}{\end{eqnarray}}
\newcommand{\beqn}{\begin{eqnarray*}}
\newcommand{\eeqn}{\end{eqnarray*}}

\hypersetup{
  pdftitle={Certified Spherical MUSIC for 3D Localization under
    Adversarial Subspace Perturbations},
  pdfauthor={Albert Fannjiang and Chuong Nguyen},
  pdfsubject={Inverse scattering, MUSIC, localization, and
    subspace perturbations},
  pdfkeywords={MUSIC, inverse scattering, point scatterers,
    localization, subspace perturbation}
}

\title[Certified spherical MUSIC]
  {Certified Spherical MUSIC for 3D Localization
   under Adversarial Subspace Perturbations}
\author{Albert Fannjiang}
\author{Chuong Nguyen}
\date{September 4, 2026}
\email{cafannjiang@ucdavis.edu}
\address{Department of Mathematics, UC Davis, CA 95616}

\begin{document}
	\begin{abstract}
	We study an oracle subspace-perturbation model for 3D localization. The procedure is given an $s$-dimensional subspace $\widetilde{\mathcal{U}}$ and a deterministic error bound
\[
\varepsilon_{\mathrm{sub}}
=
\left\|P_{\widetilde{\mathcal{U}}}-P_{\mathcal{U}}\right\|_2,
\]
where $\mathcal{U}$ is an $s$-dimensional subspace of far-field patterns with wavenumber $\kappa$, and $P_{\mathcal{V}}$ denotes orthogonal projection onto a subspace $\mathcal{V}$.

Under explicit separation and conditioning hypotheses for arbitrary point clouds, we prove that the perturbed spherical MUSIC objective
\[
\widetilde{q}(\mathbf{z})
=
1-\left\|P_{\widetilde{\mathcal{U}}}\varphi_{\mathbf{z}}\right\|_2^2
\]
has a unique strongly convex well in every ball $B_{\gamma/\kappa}(x_j)$ and a uniform value gap outside the union of the certified wells. A fixed-step gradient map with $h\asymp\kappa^{-2}$ leaves each well invariant and converges linearly to its unique minimizer. Consequently, thresholding on an $O(\kappa^{-1})$-mesh, followed by gradient descent and duplicate removal, recovers all relevant minima with localization bound
\[
\mathfrak{R}(\varepsilon_{\mathrm{sub}})
\lesssim
\frac{\varepsilon_{\mathrm{sub}}}{\kappa}.
\]

The arbitrary-cloud frame analysis yields the sufficient condition
\[
\kappa\delta_X\gtrsim s^{2/3}
\]
through an absolute coherence row sum and Gershgorin's theorem. We construct lower-frame counterexamples below the $s^{1/6}$ scale, upper-frame counterexamples below the $s^{1/3}$ scale, and examples showing that $2/3$ is optimal for the absolute-row-sum argument. Finally, for parameter classes containing a uniformly admissible one-point displacement path, we prove
\[
\mathfrak{R}(\varepsilon_{\mathrm{sub}})
\sim
\frac{\varepsilon_{\mathrm{sub}}}{\kappa}.
\]
\end{abstract}

		\maketitle

\tableofcontents

	\section{Introduction}
	
	Inverse scattering couples two problems that are mathematically
different.  The first is an acquisition problem: an incident field
interacts with an unknown object, the resulting scattered field must be
separated from the direct or background field, and a stable signal
subspace must be estimated from finite and noisy measurements.  The
second is a geometric inversion problem: once that subspace is
available, one must recover the locations of the scatterers from it.
The purpose of this paper is to isolate the second problem and give a
quantitative, algorithmic theory for spherical MUSIC.  The interface
between the two problems is a deterministic projector-perturbation
bound.

	Although much of inverse  scattering is formulated for extended obstacles or continuously distributed contrasts \cite{Cakoni,CK,kirsch02,Melrose}, the Foldy--Lax point-scatterer model is more than a simplifying surrogate: it retains nontrivial multiple-scattering interactions through the coupled effective coefficients while reducing the unknown geometry to a finite set of locations \cite{foldy45,lax51}.  This makes it a particularly transparent setting for quantitative stability analysis.
	
	 A foundational strand of inverse problem theory asks whether exact, noiseless data determine the unknown uniquely; under perturbations, however, uniqueness alone gives no control of the resulting reconstruction error \cite{nachman88,novikov94,uhlmann87,Melrose}. In the oracle subspace formulation adopted here, the corresponding question is how an error in the estimated signal subspace propagates into errors in the recovered locations. Although analogous stability questions arise for continuous media, sharp and geometry-explicit estimates that simultaneously resolve the roles of wavelength, noise, and structural complexity remain difficult in that general setting. 
	 
	 For point scatterers, by contrast, the minimum separation \(\delta_X\), together with the wavenumber \(\kappa\) and the number \(s\) of scatterers, enters explicitly into conditioning and resolvability. A general continuous contrast has no canonical point-to-point minimum separation; hence our results do not themselves give a stability theory for extended media, but instead pose a provocative question: what notion of scale separation or effective complexity plays an analogous role in the continuum, and can it be recovered as a meaningful limit of the discrete theory? In this sense, the discrete model provides both a rigorous test bed for stability and a possible source of principles for the more difficult continuous problem.
	
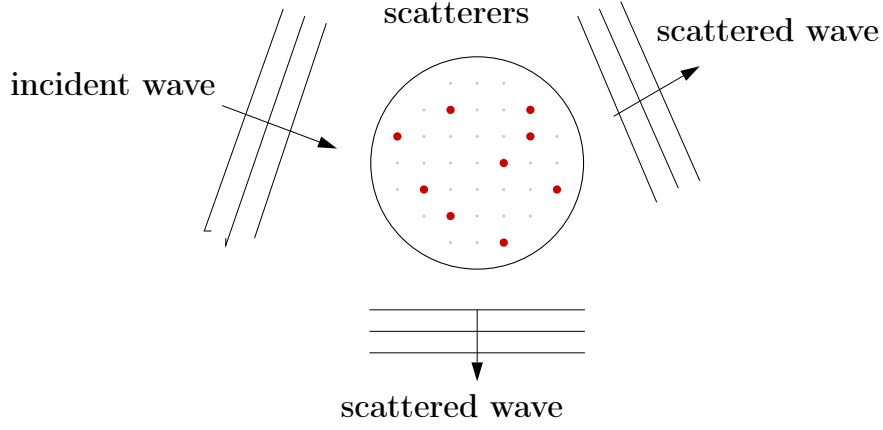
\begin{figure}[t]
\centering
\scalebox{0.5}{
\begin{tikzpicture}[x=1pt,y=1pt,line cap=butt,line join=miter]
  \path[use as bounding box] (0,0) rectangle (677,348);

  \tikzset{
    label24/.style={inner sep=0pt,outer sep=0pt,font=\fontsize{24}{28}\selectfont},
    label25/.style={inner sep=0pt,outer sep=0pt,font=\fontsize{25}{29}\selectfont},
    wavefront/.style={draw=black,line width=0.45pt},
  }

  \node[label24,anchor=base west] at (5.67,264.33) {\bfseries incident wave};
  \node[label24,anchor=base west] at (286.47,318.33) {\bfseries scatterers};
  \node[label25,anchor=base west] at (491.67,304.83) {\bf scattered wave};
  \node[label24,anchor=base west] at (254.07,21.33) {\bf scattered wave};

  \draw[wavefront] (210.87,329.13) -- (151.47,161.73) -- (156.87,161.73);
  \draw[wavefront] (227.07,323.73) -- (167.67,150.93) -- (167.67,156.33);
  \draw[wavefront] (243.27,318.33) -- (189.27,156.33);

  \draw[wavefront] (164.97,256.23) -- (237.89,228.89);
  \fill (251.37,223.83) -- (236.38,224.83) -- (239.42,232.92) -- cycle;

  \begin{scope}[shift={(356.67pt,213pt)},x=40pt,y=40pt]
    \draw[thick] (0,0) circle[radius=2.0];

    \foreach \x in {-1.5,-1.0,...,1.5} {
      \foreach \y in {-1.5,-1.0,...,1.5} {
        \pgfmathsetmacro{\r}{sqrt(\x*\x+\y*\y)}
        \ifdim\r pt<1.8pt
          \fill[gray!45] (\x,\y) circle[radius=1.1pt];
        \fi
      }
    }

    \foreach \p in {
      (-1.5,0.5),
      (-1.0,-0.5),
      (-0.5,1.0),
      (-0.5,-1.0),
      (0.5,0.0),
      (0.5,-1.5),
      (1.0,1.0),
      (1.0,0.5),
      (1.5,-0.5)
    }{
      \fill[red!80!black] \p circle[radius=3.1pt];
    }
  \end{scope}

  \draw[wavefront] (432.27,323.73) -- (491.67,183.33);
  \draw[wavefront] (448.47,329.13) -- (507.87,194.13);
  \draw[wavefront] (464.67,334.53) -- (524.07,199.53);
  \draw[wavefront] (459.27,248.13) -- (511.63,278.67);
  \fill (524.07,285.93) -- (513.82,274.93) -- (509.47,282.39) -- cycle;

  \draw[wavefront] (275.67,69.93) -- (437.67,69.93);
  \draw[wavefront] (275.67,86.13) -- (437.67,86.13);
  \draw[wavefront] (275.67,102.33) -- (437.67,102.33);
  \draw[wavefront] (356.67,102.33) -- (356.67,62.73);
  \fill (356.67,48.33) -- (352.36,62.72) -- (361.00,62.72) -- cycle;
\end{tikzpicture}
}
\caption{Scattering experiments with incident and scattered waves}
\end{figure}

	Let
\[
    X=\{\bx_1,\ldots,\bx_s\}\subset B_1(0)
\]
be a cloud of \(s\) distinct point scatterers.  
For the \(m\)-th
incident field, the far-field pattern has the form
\begin{equation}
\label{eq:intro-foldy-lax}
    u_m^\infty(\bom)
    =
    \sum_{j=1}^s
       \alpha_j^{(m)}
       e^{\ii\kappa\bom\cdot\bx_j},
    \qquad
    \bom\in\mathbb S^2,
\end{equation}
up to convention-dependent signs and normalization constants. 	In the Foldy--Lax point-scatterer model,  the
effective coefficient vector $\bm{\alpha}^{(m)}=(\alpha^{(m)}_1,\dots,\alpha^{(m)}_s)^T$  is determined by the coupled system
\begin{equation}
\label{eq:intro-effective-coefficients}
    \bm\alpha^{(m)}
    =
    A(\Id-\mathcal L_XA)^{-1}
    \bm u_{\rm inc}^{(m)},
\end{equation}
where \(A=\diag[a_1,\dots,a_s]\) contains the individual scattering strengths,
\(\mathcal L_X\) is the off-diagonal interaction matrix
$$(\mathcal L_X)_{ij}=G(\bx_i,\bx_j),\,\,i\neq j,\quad (\mathcal L_X)_{ii}=0,\,\,\forall i,$$
and
\(\bm u_{\rm inc}^{(m)}\) is the incident field evaluated at the
scatterer locations.    Here for convenience we make the non-resonance assumption that  $\Id-\mathcal L_X A$ is invertible. If incidence information is not used, $\bm{\alpha}$ is treated as a nuisance coefficient vector
\cite{Martin,Mishchenko}. 

The coefficients in \eqref{eq:intro-foldy-lax} are therefore not the
bare scattering strengths.  They contain all recurrent point-to-point
interactions represented by the Foldy--Lax system.  Thus the model
retains multiple scattering while reducing the forward problem to a
finite system.  In particular, this is not a Born or single-scattering
linearization.  Multiple scattering changes the coefficient vector
\(\bm\alpha^{(m)}\), but it does not change the family of spherical
Fourier atoms associated with the locations.

Define
\[
    \varphi_\bx(\bom)
    :=
    e^{\ii\kappa\bom\cdot\bx},
    \qquad
    \Phi_Xc
    :=
    \sum_{j=1}^s c_j\varphi_{\bx_j},
\]
and
\[
    \mathcal U_X
    :=
    \operatorname{Ran}\Phi_X
    =
    \operatorname{span}
    \{\varphi_{\bx_1},\ldots,\varphi_{\bx_s}\}
    \subset
 \cH:=   L^2\left(\mathbb S^2,\frac{d\bom}{4\pi}\right)
\]
where 	 the complex Hilbert space $\cH$ of the far-field patterns	is	equipped with   the normalized inner product
	\[
	\ip{f}{g}
	:=\frac1{4\pi}\int_{\Sph} f(\bm\omega)\overline{g(\bm\omega)}\dd\bm\omega
	\]
   and 
    the Fourier subspace $U$ has dimension $s$ iff  $\bx_i\neq \bx_j, i\neq j$.

For \(M\) scattered-field snapshots, let
\[
    C
    :=
    \begin{bmatrix}
       \bm\alpha^{(1)} & \dots & \bm\alpha^{(M)}
    \end{bmatrix}
    \in\mathbb C^{s\times M}.
\]
We assume \(\operatorname{rank}C=s\), hence sufficiently diverse effective illuminations. 
The multi-snapshot data operator then factors as
\begin{equation}
\label{eq:intro-data-factorization}
    Y=\Phi_XC.
\end{equation}
Consequently,
\begin{equation}
\label{eq:intro-range-condition}
    \operatorname{Ran}Y\subseteq\mathcal U_X ,
    \qquad
    \operatorname{Ran}Y=\mathcal U_X
    \quad\Longleftrightarrow\quad
    \operatorname{rank}C=s.
\end{equation}
The incident fields, scattering strengths, and Foldy--Lax interactions
enter \eqref{eq:intro-data-factorization} through the right-hand
coefficient matrix \(C\).  In particular, an unknown nonsingular
mixing of the snapshots does not change their signal subspace:
\[
    \operatorname{Ran}(YQ)=\operatorname{Ran}Y,
    \qquad
    Q\in\mathbb C^{M\times M}
    \quad\hbox{invertible}.
\]
This algebraic invariance is the basic reason that incidence
parameters need not enter the subsequent localization stage.

The rank condition in \eqref{eq:intro-range-condition} is essential.
It requires sufficiently diverse effective illuminations of all
scatterers; \(M\geq s\) alone does not guarantee it.  Coherent
illuminations, vanishing effective coefficients, poor conditioning
near a Foldy--Lax resonance, or an unfavorable acquisition geometry
can all prevent accurate recovery of the full \(s\)-dimensional
subspace.

The construction of a stable approximation to \(\mathcal U\) and
the inversion of that subspace are governed by different mechanisms.
The former depends on illumination diversity, aperture, receiver
calibration, direct-field rejection, multiple scattering, and the
measurement-noise model.  The latter depends on the geometry of the
spherical Fourier atoms, the separation and arrangement of the points,
and the optimization landscape of the MUSIC objective.

We separate these issues by taking as input an \(s\)-dimensional
subspace \(\widetilde{\mathcal U}\), not necessarily itself a Fourier
subspace, together with the deterministic certificate
\begin{equation}
\label{eq:intro-oracle-bound}
    \varepsilon_{\rm sub}
:=   \|P_{\widetilde{\mathcal U}}-P_{\mathcal U_X}\|_2
    <1
\end{equation}
where $\Proj_{\cU}$ and $\Proj_{\widetilde\cU}$  denote the orthogonal projector onto $\cU$ and $\widetilde\cU$, respectively.  
Because the two subspaces have the same finite dimension,
\(\varepsilon_{\rm sub}\) is the sine of their largest principal
angle.  The upstream acquisition and subspace-estimation problem is
compressed into the single quantity \(\varepsilon_{\rm sub}\).  The
downstream localization problem uses only
\(P_{\widetilde{\mathcal U}}\), the wavenumber \(\kappa\), and explicit
geometric and algorithmic parameters.

In other words, our method is
\emph{incidence blind}: after an isolated scattered-field subspace has
been supplied, the localization procedure does not require the
directions, amplitudes, phases, or effective Foldy--Lax coefficients of
the illuminations. 

In the sequel, we write $A\lesssim B$ if $A\le C B$ for an absolute constant $C$, and $A\asymp B$ if both $A\lesssim B$ and $B\lesssim A$. $A\gtrsim B$ is likewise defined.  Constants may depend on dimensionless parameters such as the parameter $\gamma$, but not on $\kappa$, $\delta_X$, $s$, or the point configuration, unless explicitly stated.

	 \subsection{Main contribution}	 
 A standard MUSIC calculation without an accompanying
model-specific theorem evaluates a pseudospectrum on a search grid
and reports its peaks or smallest residual values, but does not by
itself certify that every target has been
initialized, that no accepted point lies outside the target
neighborhoods, that a chosen local iteration remains in the correct
neighborhood, or that it converges at a controlled rate \cite{fannjiang11}. 

For the continuous spherical oracle model considered here, we obtain
a single quantitative certificate that combines all four properties.
 
\begin{enumerate}[label=\textup{(\roman*)}]
\item
Under explicit boundary-containment, arbitrary-cloud separation,
frame-conditioning, and subspace-perturbation hypotheses, the
objective \(\widetilde q\) has a
target-associated local minimum \(\widetilde\bx_j\) in every certified
neighborhood
$
    B_{\rho}(\bx_j),\,\,\rho\asymp\kappa^{-1}, 
$ and
\begin{equation}
\label{eq:intro-localization-rate}
    |\widetilde\bx_j-\bx_j|
  \lesssim    \frac{\varepsilon_{\rm sub}}{\kappa}.
\end{equation}

In addition, the MUSIC-specific local
geometry satisfies 
$    \nabla^2\widetilde q(\bz)\asymp\kappa^2 I_3$ and  each certified neighborhood is a strongly convex well
containing exactly one critical point.

\item 
A finite grid \(G\subset B_1(0)\)  of spacing $O(\kappa^{-1})$ 
generates by a thresholding  rule  at least one grid point in every certified neighborhood and
accepts no grid point outside their union, which can serve as initialization for gradient descent with explicitly certified linear convergence rate.

\item
For arbitrary \(s\)-point clouds, an absolute coherence row sum and
Gershgorin's theorem give the sufficient frame-conditioning scale
\[
    \kappa\delta_X
    \gtrsim
    s^{2/3}.
\]
\item Finally, on uniform configuration classes satisfying the
landscape margins, the deterministic oracle localization modulus
obeys
\[
    \mathfrak R(\eta)
    \lesssim
    \frac{\eta}{\kappa}.
\]
For classes that additionally contain a uniformly admissible
one-point displacement path, a matching two-point lower bound gives
\[
    \mathfrak R(\eta)
    \asymp
    \frac{\eta}{\kappa}. 
\]
\end{enumerate}

Restricted to a uniformly separated and conditioned class, the map sending a point configuration \(X\) to its Fourier subspace \(\mathcal U\) is locally bi-Lipschitz, in the sense that the Grassmannian projector distance is comparable to \(\kappa d_{\rm match}(X,Y)\) where 
\begin{equation}
\label{eq:oracle-matching-distance}
    d_{\rm match}(X,Y)
    :=
    \min_{\pi\in S_s}
    \max_{1\leq j\leq s}|\bx_j-\by_{\pi(j)}|. 
\end{equation}

  The admissible-landscape, coarse-thresholding, and gradient-refinement
mechanism is adapted from the multidimensional Gradient--MUSIC
framework of \cite{gradient-music}. Relative to that abstract
framework, the present analysis focuses on the full-aperture spherical
kernel arising from the Foldy--Lax point-scatterer model, explicit
arbitrary-cloud frame estimates, obstruction examples for the
resulting separation scales, and a deterministic Grassmannian oracle
localization modulus.

\subsection{Related literature} The relationship between MUSIC, the linear sampling method, and the
factorization method is classical
\cite{cheney01,kirsch02,arens09,Cakoni}.  Rigorous stability and
super-resolution analyses for sparse Fourier MUSIC include
\cite{fannjiang11,liao15,liao-fannjiang16}.  In inverse scattering,
subspace and time-reversal methods have also been developed for
point targets with significant target-to-target multiple scattering
\cite{devaney05,marengo-gruber07,challa-sini12}.  More recently,
Alberti, Petit, and Santacesaria \cite{alberti24} proposed a
linearize-and-locally-optimize method for nonlinear Foldy--Lax
observations.

The literature most closely related to this work addresses different
stages of the inverse problem: the physical map from scatterers and
illuminations to measured fields, the estimation of a signal subspace
from those fields, and the subsequent recovery of locations from that
subspace.  Since these stages use different data metrics, their
stability rates are not directly comparable without an additional
result linking measurement error to subspace error.

At the level of qualitative range tests, Cheney \cite{cheney01}
explained the relationship between point-target MUSIC and the linear
sampling method.  Arens, Lechleiter, and Luke \cite{arens09} subsequently
placed MUSIC for extended nonabsorbing scatterers and cracks within the
factorization-method framework, including stability considerations for
noisy and limited data.  These theories apply to classes of obstacles
and media considerably broader than finite point clouds; see also
\cite{Cakoni,kirsch02}.  Our restriction to a finite Foldy--Lax cloud is
therefore not intended to enlarge the qualitative applicability of
MUSIC.  It permits us to exploit the explicit full-aperture spherical
kernel in order to obtain quantitative curvature, frame-conditioning,
initialization, optimization, and point-localization estimates.

Sparse-recovery and spectral-estimation theories provide further close
precedents.  Fannjiang \cite{fannjiang10} analyzed high-frequency
SIMO, MISO, and MIMO inverse scattering of discrete targets by
compressed sensing, both with and without the Born approximation, and
derived probabilistic recoverability results for
\(\ell^1\)-based minimization and regularization.  
A particularly relevant comparison is the high-frequency compressed-sensing analysis of Fannjiang \cite{fannjiang10}  which derives for point scatterers constrained to a fixed grid of spacing \(\ell\) the uniform noisy-recovery condition
\(\kappa\ell\gtrsim s\) for the multiple-scattering Foldy-Lax model in contrast to our scaling law
$\kappa\delta_X\gtrsim s^{2/3}$ obtained without incidence information and grid-constraint.

Fannjiang
\cite{fannjiang11} subsequently obtained coherence- and
restricted-isometry-based guarantees for sparse-object MUSIC,
including noisy thresholding, refined-grid localization, and a
multiple-scattering extension.  Off-grid stability and
super-resolution for one- and multidimensional spectral MUSIC were
developed in \cite{liao-fannjiang16,liao15}.  Thus noisy recovery,
off-grid localization, multidimensionality, and multiple scattering
are established features of earlier sparse-imaging and MUSIC theory.
The present setting differs in its continuous full-aperture spherical
data space and in its deterministic oracle input.

A related but different ``linearize and locally
optimize'' strategy was proposed by Alberti, Petit, and Santacesaria
\cite{alberti24} for nonlinear Foldy--Lax observations.  In their method,
a Beurling--LASSO problem based on the Born linearization supplies an
initialization, after which the nonlinear forward functional is locally
minimized.  Their theoretical guarantees concern the initial
sparse-measure reconstruction under small-intensity and separation
hypotheses, together with numerical evidence for the subsequent
nonlinear refinement. 

The result should also be distinguished from stability estimates for
the full parameter-to-data map.  Bourgeois \cite{bourgeois13} proved
inverse Lipschitz stability on compact finite-dimensional parameter
sets, whereas Hohage and Weidling \cite{hohage15} obtained logarithmic
regularization rates for Tikhonov regularization of an infinite-dimensional refractive index.
For one fixed incident wave and its full complex far-field pattern,
Rondi \cite{rondi08} proved nested-logarithmic and sharper
obstacle-only conditional moduli for polyhedral scatterers, while
Sincich and Sini \cite{SS08} proved logarithmic stability on a
geometrically restricted class of sound-soft obstacles.  Both results compare exact, model-generated
far-field patterns and provide conditional moduli of injectivity; they
do not obtain reconstruction estimators for arbitrary noisy data.

Inverse point-source results are geometrically closer to the recovery
of a finite point cloud, but they concern a different physical map.
Mdimagh and Ben Saad \cite{MBS16} established local Lipschitz stability
for point sources and  Bao, Liu, and Triki \cite{bao21} proved
H\"older stability for separated radiative point sources in a known
inhomogeneous medium, both  from a single
boundary Cauchy measurement.

  \subsection{Organization}
  We present our main results,  including the landscape theorem and localization modulus theorem,  in \cref{sec:main}. 
 We state and prove a uniform frame bounds in \cref{sec:frame}, a key property of  independent interest. 
\cref{sec:energy}   establishes kernel and
nonlocal-energy estimates. 

We prove the main theorems in \cref{sec:proof}.
  We present and certify the initialization-augmented optimization-formulated MUSIC algorithm in \cref{sec:algorithm} and conduct a numerical experiment in \cref{sec:num}. 
  
  We give an elementary proof of  the positivity of the Gram matrix in \cref{app:gram}
and  lower-frame, upper-frame, and absolute-row-sum constructions in \cref{app:lower-example}, \cref{app:upper-example}, \cref{app:row-sum-example}
that clarify the sharpness and limitations of the arbitrary-cloud
separation analysis.

\section{Main Theorems}\label{sec:main}
	The noiseless and perturbed MUSIC functions are
	\[
	q_X(\bm{z})=\norm{(\Id-\Proj_{\cU_X})\varphi_{\bm{z}}} ^2,
	\qquad
	\widetilde q(\bm{z})=\norm{(\Id-\Proj_{\widetilde{\cU}})\varphi_{\bm{z}}} ^2,
	\]
	\cite{schmidt86,kirsch02}. 
	Since $\varphi_{\bm{x}_j}\in\cU_X$, one has
	\[
	q_X(\bm{x}_j)=0,
	\qquad j=1,\dots,s.
	\]
On the other hand, since $\widetilde U$ is generally not a Fourier subspace, $\widetilde q$ may not have any roots.

	Define the measurement kernel as 
	\begin{eqnarray}
		\label{eq:spherical-kernel} 
		K_\kappa(\bm{h})
		:=\inner{\varphi_{\bm{0}}}{\varphi_{\bm{h}}} 
		=\frac1{4\pi}\int_{\bmS^2} e^{i\kappa\bm{\omega}\cdot \bm{h}}\dd \bm{\omega}.
	\end{eqnarray}
	By rotational invariance,
	\[
		K_\kappa(\bm{h})=
		\frac{\sin(\kappa\abs{\bm{h}})}{\kappa\abs{\bm{h}}}
	\]
	with the convention $K_\kappa(\bm{0})=1$.
	
Define the kernel vector
	\begin{eqnarray}
	\bm{k}_X(\bm{z})=
	\begin{bmatrix}
		K_\kappa(\bm{z}-\bm{x}_1)\\
		\vdots\\
		K_\kappa(\bm{z}-\bm{x}_s)
	\end{bmatrix}
	\in\bmC^s.\label{eq:kernel-vector}
	\end{eqnarray}
		Let 
	\[
	\Phi_X=\big[\varphi_{\bx_1},\dots,\varphi_{\bx_s}\big]
	\]
	denote the Vandermonde operator from $\CC^s$ to $\cH$. 
	
Let  $\Phi_X^*$ denote the adjoint operator of $\Phi_X$ from $\cH$ to $\CC^s$ and
	\beq
	\label{eq:gram}	
	G_X:=\left[K_\kappa(\bm{x}_j-\bm{x}_k)\right]_{j,k=1}^s= \Phi_X^*\Phi_X \eeq
be  the Gram matrix  which is  positive-definite and invertible  for $s$ distinct atoms (see \cref{app:gram} for an elementary proof). 
We
 can write 
	\beq
	\Proj_{\cU_X}&=&\Phi_XG_X^{-1}\Phi_X^*,\\
		q_X(\bm{z})&=&1-\bm{k}_X(\bm{z})^*G_X^{-1}\bm{k}_X(\bm{z}).\label{eq:music2}
		\eeq

The following proposition  provides a geometric interpretation of the MUSIC objective in connection to the chordal distance on the Grassmannian of $s$-dimensional subspaces in $\cH$. 

    \begin{proposition}[MUSIC as a frame-weighted chordal residual]
\label{prop:music-chordal-residual}
Let \(Z=\{\bz_1,\ldots,\bz_s\}\) be a configuration for which
$
    G_Z:=\Phi_Z^*\Phi_Z
$
is invertible, and set
\[
    Q_Z:=\Phi_ZG_Z^{-1/2}.
\]
Then \(Q_Z^*Q_Z=I_s\) and
\[
    \sum_{j=1}^s\widetilde q(\bz_j)
    =
    \|(I-P_{\widetilde{\mathcal U}})\Phi_Z\|_{\mathrm{HS}}^2.
\]
If
\[
    (1-\mu)I_s
    \preceq
    G_Z
    \preceq
    (1+\mu)I_s,\qquad 0<\mu<1,
\]
then
\beq
    (1-\mu)
    d_{\mathrm{ch}}^2(
        \mathcal U_Z,\widetilde{\mathcal U}
    )
    \leq
    \sum_{j=1}^s\widetilde q(\bz_j)
    \leq
    (1+\mu)
    d_{\mathrm{ch}}^2(
        \mathcal U_Z,\widetilde{\mathcal U}
    ),
\eeq
where
\[
    d_{\mathrm{ch}}^2(
        \mathcal U_Z,\widetilde{\mathcal U}
    )
    :=
    \|(I-P_{\widetilde{\mathcal U}})Q_Z\|_{\mathrm{HS}}^2
    =
    \frac12
    \|P_{\mathcal U_Z}
       -P_{\widetilde{\mathcal U}}\|_{\mathrm{HS}}^2.
\]
\end{proposition}

\begin{proof}
Let
\[
    A:=
    Q_Z^*(I-P_{\widetilde{\mathcal U}})Q_Z
    \succeq0.
\]
Since \(\Phi_Z=Q_ZG_Z^{1/2}\),
\[
    \sum_{j=1}^s\widetilde q(\bz_j)
    =
    \operatorname{tr}(G_ZA),
\]
whereas
\[
    d_{\mathrm{ch}}^2(
        \mathcal U_Z,\widetilde{\mathcal U}
    )
    =
    \operatorname{tr}(A).
\]
The result follows from the spectral bounds on \(G_Z\).
\end{proof}

This shows that the separable MUSIC objective is uniformly
equivalent to chordal subspace fitting on frame-conditioned candidate
configurations. Its computational advantage is that, once distinct
initial points have been placed in the certified wells, the
refinement variables decouple. The equivalence alone does not enforce
distinctness and does not by itself prove minimax optimality; those
properties follow from the threshold, well-separation, clustering,
and localization-modulus results.
	
	To lighten the notation, we will drop the subscript $X$ and write $\Phi, U, \bm k,\dots$ etc. when confusion does not arise in the sequel. 
	\subsection{Threshold-admissible  landscape theorem} \label{sec:landscape}

	Our first goal is to prove that 
	under an explicit arbitrary-cloud packing condition and an oracle-perturbation bound, the perturbed MUSIC function is a threshold-admissible landscape and  has one strongly convex well near every true point, a uniform objective gap outside the certified wells, and an \(O(\eps_{\rm sub}/\kappa)\) localization error. A fixed-step gradient map with \(h\asymp\kappa^{-2}\) is invariant and contractive on each well, and an \(O(\kappa^{-1})\)-mesh thresholding procedure supplies at least one valid initialization for every well. 
				\begin{theorem}[Certified spherical MUSIC landscape and refinement]
		\label{thm:certified-music-landscape}
			Let
		$$
		X=\{\bm{x}_1,\dots,\bm{x}_s\},\quad s\ge 2
		$$
		be an  $s$-point cloud and 
			\[
	\delta_X:=\min_{j\ne k}\abs{\bm{x}_j-\bm{x}_k}
	\]
	be its minimum separation.

Suppose		$$
		\bigcup_{j=1}^s\overline{B_\rho(\bm{x}_j)}
		\subset B_1(\bm{0})\subset\mathbb R^3, \quad\rho\le \frac{\delta_X}{2}.	$$
 Define two  fundamental dimensionless constants:
		\beq
		\label{eq:packing}
	\gamma:=\rho{\kappa}, \qquad \mu_X
		:=
		\frac{9}{2}
		\frac{s^{2/3}}{\kappa\delta_X}<1
\eeq
with  $\kappa>0$ 
	and four other derived constants
	\beq
\nu_X
		&:=&
	 \frac{64c_\nu^2}{27}
\frac{\mu_X^2}{s} = 	48c_\nu^2
		\frac{s^{1/3}}{\kappa^2\delta_X^2} \label{eq:mu-nu},\qquad c_\nu:=\sup_{t>0}\abs{t\cos t -\sin t}/t \,\,\approx 1.0631,\\
				\nonumber	C_1&:=&
			\frac{\gamma^2}{9}
			+
			\nu_X+\sqrt{1+\nu_X}
			\sqrt{				\left(
				\frac{1}{3}+\frac{\gamma^2}{10}
				\right)^2
				+
				4\nu_X
			}\\
			\nonumber 	C_2&:=& \frac{8}{15} \gamma^2 +6\nu_X+ \frac{2\mu_X}{1-\mu_X} C_1.
\eeq
Let \(\widetilde{\mathcal U}\) be any $s$-dimensional subspace of $\cH$ 
satisfying
\[
    \|P_{\widetilde{\mathcal U}}-P_{\mathcal U_X}\|_2 =\epssub.
\]
Assume 
\beq
 \label{eq:A-gamma}
 A_\gamma& :=& \frac{b_\gamma-\mu_X-\nu_X}{1-\mu_X}>0,\quad b_\gamma
		:=
		1-\sup_{t\ge \gamma} \sinc^2(t)
 \eeq
			and 	\beq \label{eq:noise}C_2<\frac23,\qquad
\epssub
	< \min\left\{ \frac{\frac{\sqrt3}{2}\gamma
            \left(\frac23-C_2\right)
       }{1+\frac{\sqrt3}{2}\gamma \left(\frac{2}{\sqrt5}+\frac23\right)},\,\,   \sin\left(
        \frac12\arcsin\sqrt{A_\gamma}.
    \right)\right\},
		\eeq
		Define
\beq
\label{eq:threshold-parameters}
\tau_0:=\epssub^2,\qquad    \tau_1
    :=
    \sin^2\left(
        \arcsin\sqrt{A_\gamma}-\arcsin\epssub
    \right).
\eeq
Then  \(\tau_0<\tau_1\) and:

\begin{enumerate}[label=\textup{(\roman*)}]
\item
For every \(j\),
\beq
\label{eq:main-hessian-bounds}
    C_0^-\kappa^2I_3
    \preceq
    \nabla^2\widetilde q(\bz)
    \preceq
    C_0^+\kappa^2I_3,
    \qquad
    \bz\in\overline{B_\rho(\bx_j)}
\eeq
where
\beq\label{eq:C_0}
C^-_0&:=&	\frac23
			-
			C_2
			-
			\left(
			\frac{2}{\sqrt5}+\frac23\right) 	\eps_{\rm sub} > \frac{
            \frac23-C_2
       }{1+\frac{\sqrt3}{2}\gamma \left(\frac{2}{\sqrt5}+\frac23\right)},\\
    C_0^+
    &:=&\frac23+C_2
      +\left(\frac2{\sqrt5}+\frac23\right)\epssub.
\eeq

\item
Every point outside the certified wells satisfies
\beq\label{eq:upper-threshold}
    \widetilde q(\bz)\geq\tau_1,
    \qquad
    \bz\in
    B_1(0)\setminus
    \bigcup_{j=1}^sB_\rho(\bx_j).
\eeq
\item
Each \(B_\rho(\bx_j)\) contains exactly one critical point
\(\widetilde\bx_j\), which is its unique minimizer, and
\beq
\label{eq:localization-error}
    |\widetilde\bx_j-\bx_j|   \le  \frac{2\epssub}{\sqrt3\,C_0^-\kappa}<\rho,
    \qquad
    \widetilde q(\widetilde\bx_j)\leq\tau_0.
\eeq

\item 	If the  fill distance of  a finite set \(G\subset B_1(0)\)
	\begin{equation}
	\label{eq:grid-fill-distance}
		\mesh_{B_1}(G)
		:=
		\sup_{\bz\in B_1(0)}
		\min_{\bm g\in G}|\bz-\bm g|
	\end{equation}
 satisfies
\[
    \operatorname{mesh}_{B_1}(G)
    <
    \min\left\{
        \rho- \frac{2\epssub}{\sqrt3\,C_0^-\kappa},
        \,
        \frac{\sqrt3}{2\kappa}(\tau_1-\tau_0)
    \right\},
\]
then
\[
    G_{\tau_1}:=
    \{\bg\in G:\widetilde q(\bg)<\tau_1\}
\]
contains at least one point in every certified well and no point
outside their union.

\item
For
\[
    0<h\leq
    \frac{2}{(C_0^-+C_0^+)\kappa^2}
    =
    \frac{3}{2\kappa^2},
\]
the gradient map
\[
    T_h(\bz):=\bz-h\nabla\widetilde q(\bz)
\]
maps every \(\overline{B_\rho(\bx_j)}\) into \(B_\rho(\bx_j)\) and is
a contraction there with factor
\[
    r_h:=1-hC_0^-\kappa^2<1.
\]
Consequently, every accepted initialization converges geometrically
to the unique minimizer in its well.
\end{enumerate}
	\end{theorem}
Proof of \cref{thm:certified-music-landscape} is given in \cref{sec:proof} after a substantial preparation in \cref{sec:frame} \& \cref{sec:energy} which also give rise to the fundamental packing condition \eqref{eq:packing}. 

The implicit inequalities \eqref{eq:A-gamma} \& \eqref{eq:noise} admit a
simple rational specialization with moderate noise allowance.

\begin{corollary}
\label{cor:concrete-parameter-regime}
Let $s\geq2$ and
\(\bigcup_{j=1}^s\overline{B_\rho(\bx_j)}
\subset B_1(0).\)
Set
\[
      \gamma:=\kappa\rho=\frac7{12},
\]
and suppose
\begin{equation}
\label{eq:concrete-separation-noise-regime}
     \mu_X\leq\frac1{20},  \qquad 0\leq\eps_{\rm sub}\leq\frac7{60}.
\end{equation}
Then
the geometric condition \(\rho\leq\delta_X/2\) holds automatically,
and the conditions \eqref{eq:A-gamma} and \eqref{eq:noise} are satisfied.
\end{corollary}
The verification of \cref{cor:concrete-parameter-regime} is elementary and tedious so will be omitted here.

\subsection{Localization modulus theorem}\label{sec:minimax}

Let \(\mathcal X_s(B_1)\) be the set of unordered \(s\)-point subsets
of \(B_1(0)\), equipped with the bottleneck matching distance $d_{\rm match}$. 
Fix \(s\geq2\), \(\kappa>0\), \(\gamma>0\), and
\(0\leq\overline\mu<1\), and put \(\rho=\gamma/\kappa\).  Define
the uniform configuration class
\begin{equation}
\label{eq:uniform-oracle-class}
\begin{split}
    \mathfrak X_{s,\kappa}(\gamma,\overline\mu)
    :=\biggl\{X\in\mathcal X_s(B_1):\;&
       \bigcup_{\bx\in X}\overline{B_\rho(\bx)}\subset B_1(0),\,\,\rho\leq\frac{\delta_X}{2},
       \quad \mu_X\leq\overline\mu
    \biggr\}.
\end{split}
\end{equation}
The constants required for uniformization are
\begin{align}
\notag
    \overline\nu
    &:=\frac{64c_\nu^2}{27}\frac{\overline\mu^2}{s},\\
    \overline C_1
    &:=\frac{\gamma^2}{9}+\overline\nu
       +\sqrt{1+\overline\nu}
        \sqrt{\left(\frac13+\frac{\gamma^2}{10}\right)^2
                    +4\overline\nu},\notag\\
    \overline C_2
    &:=\frac{8\gamma^2}{15}+6\overline\nu
       +\frac{2\overline\mu}{1-\overline\mu}\overline C_1,\notag\\
    A_\star
    &:=\frac{b_\gamma-\overline\mu-\overline\nu}
             {1-\overline\mu}.\notag
\end{align}
Set
\beqn
\nonumber
    \eta_{\rm loc}
    &:=&\frac{\frac{\sqrt3}{2}\gamma
            \left(\frac23-\overline C_2\right)
       }{1+\frac{\sqrt3}{2}\gamma \left(\frac{2}{\sqrt5}+\frac23\right)},\\
    \eta_{\rm gap}
    &:=&\sin\left(\frac12\arcsin\sqrt{A_\star}\right),\\
    \eta_\star&:=&\min\{\eta_{\rm loc},\eta_{\rm gap}\}.
  \eeqn
   Assume analogously to \eqref{eq:A-gamma}
\begin{equation}
\label{eq:oracle-base-feasibility}
    \overline C_2<\frac23,
    \qquad
    A_\star>0.
\end{equation}
  It is straightforward to verify that if $\epssub<\eta_\star$, then \eqref{eq:noise} holds. 

\begin{corollary}[Local bi-Lipschitz stability]
\label{cor:local-bi-lipschitz}
Let
\[
    d_{\rm sp}(X,Y)
    :=
    \|P_{\mathcal U_X}-P_{\mathcal U_Y}\|_2,\qquad 
  \forall  X,Y\in
    \mathfrak X_{s,\kappa}
    (\gamma,\overline\mu).
\]

Then
\[
    d_{\rm sp}(X,Y)
    \leq
    \frac{\kappa\sqrt{s}}
         {\sqrt{3(1-\overline\mu)}}
    d_{\rm match}(X,Y).
\]
If, in addition, \(d_{\rm sp}(X,Y)<\eta_\star\), then
\[
    d_{\rm match}(X,Y)
    \leq
    \frac{2}
         {\sqrt3\,c_{\rm sp}(X,Y)}
    \frac{d_{\rm sp}(X,Y)}{\kappa},
\]
where
\[
    c_{\rm sp}
    :=
    \frac23-\overline C_2
    -
    \left(\frac2{\sqrt5}+\frac23\right)d_{\rm sp}.
\]
Consequently, the configuration-to-subspace map is locally
bi-Lipschitz, with constants depending only on
\((\gamma,\overline\mu,s)\).
\end{corollary}

\begin{proof}
Choose a bottleneck-optimal labeling of \(X\) and \(Y\).
The synthesis-operator perturbation estimate gives
\[
\begin{aligned}
    d_{\rm sp}(X,Y)
    &\leq
    \frac{\|\Phi_X-\Phi_Y\|_2}
         {\sqrt{1-\overline\mu}}
    \\
    &\leq
    \frac{\|\Phi_X-\Phi_Y\|_{\mathrm{HS}}}
         {\sqrt{1-\overline\mu}}
    \\
    &\leq
    \frac{\kappa}
         {\sqrt{3(1-\overline\mu)}}
    \left(
        \sum_{j=1}^s
        |\bx_j-\by_j|^2
    \right)^{1/2},
\end{aligned}
\]
which proves the forward inequality.

For the inverse inequality, apply
\cref{thm:certified-music-landscape} to the truth \(X\) and
the observed subspace
\(\widetilde{\mathcal U}=\mathcal U_Y\), whose exact
projector error is \(d_{\rm sp}(X,Y)\).  For every \(\by_j\in Y\),
$
    \widetilde q(\by_j)=0.
$
The exterior threshold therefore places every \(\by_j\) in
one of the certified \(X\)-wells.  Strong convexity implies
that a well contains at most one zero of \(\widetilde q\).
Since there are \(s\) points and \(s\) wells, every well
contains exactly one point of \(Y\), and that point is its
unique minimizer.  The localization estimate in
\cref{thm:certified-music-landscape} now gives the claimed
inverse inequality.
\end{proof}

We next quantify the best possible localization accuracy when the
only observation is a subspace of dimension \(s\), lying within a
prescribed projector distance of the true Fourier subspace.  The uniform class
must control all hypotheses used in the landscape theorem.  

For a nonempty subclass
\(\mathfrak X\subset
  \mathfrak X_{s,\kappa}(\gamma,\overline\mu)\), define its
localization modulus by
\begin{equation}
\label{eq:oracle-risk-definition}
    \mathfrak R_{\mathfrak X}(\eta)
    :=
    \inf_{\widehat X}
    \sup_{X\in\mathfrak X}
    \sup_{\substack{
       \widetilde{\mathcal U}\in\operatorname{Gr}_s(\mathcal H)\\
       \|P_{\widetilde{\mathcal U}}-P_{\mathcal U_X}\|_2\leq\eta
    }}
    d_{\rm match}\bigl(\widehat X(\widetilde{\mathcal U}),X\bigr),
\end{equation}
where the infimum is over all maps
\(\widehat X:\operatorname{Gr}_s(\mathcal H)\longrightarrow
\mathcal X_s(B_1)\). 

\begin{theorem}[Localization modulus]
\label{thm:localization-modulus}
Under \eqref{eq:oracle-base-feasibility}, every nonempty
\(\mathfrak X\subset
\mathfrak X_{s,\kappa}(\gamma,\overline\mu)\) satisfies
\begin{equation}
\label{eq:oracle-upper-bound}
    \mathfrak R_{\mathfrak X}(\eta)
    \leq
    \frac{2}{\sqrt3c_\eta}\frac{\eta}{\kappa},
    \qquad 0\leq\eta<\eta_\star,
\end{equation}
where  $$  c_\eta
    :=\frac23-\overline C_2-\left(\frac{2}{\sqrt5}+\frac23\right)\eta.$$

Suppose, in addition, that there exist
\(X_0=\{\bx_1,\ldots,\bx_s\}\in\mathfrak X\), a unit vector
\(\bv\in\mathbb R^3\), and \(t_0>0\) such that
\begin{equation}
\label{eq:oracle-displacement-path}
    X_t:=\{\bx_1+t\bv,\bx_2,\ldots,\bx_s\}\in\mathfrak X,
    \qquad 0\leq t\leq t_0,
    \qquad t_0\leq\frac{\delta_{X_0}}4.
\end{equation}
Then
\begin{equation}
\label{eq:oracle-two-sided-bound}
    \frac{\sqrt{3(1-\bar\mu)}}2\frac{\eta}{\kappa}
    \leq
    \mathfrak R_{\mathfrak X}(\eta)
    \leq
    \frac{2}{\sqrt3c_\eta}\frac{\eta}{\kappa}
\end{equation}
whenever
\begin{equation}
\label{eq:oracle-two-sided-range}
    0<\eta<\eta_\star,
    \qquad
    \eta\leq\frac{\kappa t_0}{\sqrt{3(1-\overline\mu)}}.
\end{equation}
Consequently, on every uniform class containing such an interior
one-point displacement path,
\[
    \mathfrak R_{\mathfrak X}(\eta)
    \asymp_{\bar\mu, s}\frac{\eta}{\kappa}
    \qquad(\eta\downarrow0).
\]
\end{theorem}

The proof of \cref{thm:localization-modulus} is given in \cref{sec:localization-modulus}. 
		\section{Frame bounds}\label{sec:frame}
			
		To prove the admissibility of the MUSIC function, we require a uniform frame bounds, equivalently a Riesz sequence inequality, for the spherical Fourier atoms associated with the scatterer locations. The proof is based on a sparse spherical packing analysis.
			\begin{theorem}\label{thm:sparse-spherical-frame}
						Let
		$$
		X=\{\bm{x}_1,\dots,\bm{x}_s\}\subset B_1(0),\quad s\ge2.
		$$
			If
		\begin{equation}\label{eq:sparse-packing}
			\mu_X:=\frac92\frac{s^{2/3}}{\kappa\delta_X}<1,
		\end{equation}
		then for every $\bm \alpha=(\alpha_1,\ldots,\alpha_s)\in\C^s$,
		\begin{equation}\label{eq:sparse-frame}
			(1-\mu_X)\norm{\bm \alpha}^2
			\le
			\frac1{4\pi}\int_{\Sph}
			\left|\sum_{j=1}^s \alpha_j e^{i\kappa\bm\omega\cdot\bm x_j}\right|^2\dd\bm\omega
			\le
			(1+\mu_X)\norm{\bm \alpha}^2.
		\end{equation}
		Equivalently, the Gram matrix $G_X$ satisfies
		\begin{equation}\label{eq:eigenvalue-bounds}
			1-\mu_X
			\le
			\lambda_{\min}(G_X)
			\le
			\lambda_{\max}(G_X)
			\le
			1+\mu_X.
		\end{equation}
	\end{theorem}
	\begin{remark} In contrast to the sufficient condition \eqref{eq:sparse-packing}, 
\cref{app:lower-example} shows that $\kappa\delta_X\gtrsim s^{1/6}$ is the necessary worst-case packing condition for the {\bf lower}-frame bound while \cref{app:upper-example} shows that $\kappa\delta_X\gtrsim s^{1/3}$ is the necessary worst-case packing condition for the {\bf upper}-frame bound, both by analyzing grid-restricted configurations. 

Therefore, a two-sided arbitrary-cloud spectral theorem needs at least the \(s^{1/3}\) scale in the worst case.
The interval between exponents \(1/3\) and \(2/3\) remains open for two-sided spectral conditioning.

\end{remark}

	\begin{remark}
	 A $\delta_X$-separated set in $B_1(0)\subset\R^3$ can have at most
	$
	s\lesssim \delta_X^{-3}
	$
	points.  If the point cloud is volumetrically saturated, then
	\[
	s\asymp \delta_X^{-3}.
	\]
	Therefore
	\[
	s^{2/3}\asymp \delta_X^{-2},
	\]
	and the sparse condition $\mu_X<1$
		becomes
	\[
	\kappa\delta_X\gtrsim \delta_X^{-2},
	\]
	which is equivalent to
	\beq
	\label{eq:separation-condition}
	\kappa\delta_X^3\gtrsim 1.
	\eeq
\end{remark}
	
\subsection{Coherence row sum estimate}	
	
	\begin{lemma}\label{lem:sparse-row-sum}
		The coherence row sum defined as
		\[
		\sigma_X:=
		\max_{1\le j\le s}
		\sum_{k\ne j}\abs{K_\kappa(\bm{x}_j-\bm{x}_k)}
		\]
		satisfies the bound
		\begin{equation}\label{eq:sparse-row-sum}
			\sigma_X       \le \mu_X. 
		\end{equation}
	\end{lemma}

	\begin{remark}
	 \cref{app:row-sum-example}  proves that the exponent \(2/3\) is optimal for the absolute coherence row-sum estimate. Consequently, the \(s^{2/3}\) scaling cannot be improved within the present absolute-row-sum/Gershgorin argument approach to the spherical frame bounds. No spectral necessity or sharpness of the constant \(9/2\) is claimed.

\end{remark}
	
	\begin{proof}
		For $j\ne k$, the kernel bound from \eqref{eq:spherical-kernel} gives
		\[
		\abs{K_\kappa(\bm x_j-\bm x_k)}
		=\left|\frac{\sin(\kappa\abs{\bm x_j-\bm x_k})}{\kappa\abs{\bm x_j-\bm x_k}}\right|
		\le
		\frac1{\kappa\abs{\bm x_j-\bm x_k}}.
		\]

Next we need the following sparse packing estimate. 		
			\begin{lemma}\label{lem:distance-ordering}
		Fix $j\in\{1,\ldots,s\}$ and order the remaining points by increasing distance from $\bm x_j$:
		\[
		r_1\le r_2\le\cdots\le r_{s-1},
		\qquad
		r_n=\abs{\bm x_j-\bm x_{k_n}}.
		\]
		Then 
		\begin{equation}\label{eq:rn-lower}
			r_n\ge \frac13\delta_X n^{1/3},
			\qquad n=1,\ldots,s-1.
		\end{equation}
	\end{lemma}
	
	\begin{proof}
		The balls
		\[
		B(\bm x_{k_m},\delta_X/2),
		\qquad m=1,\ldots,n,
		\]
		are disjoint by the separation assumption.  If $r_n$ is the distance from $\bm x_j$ to the $n$-th nearest point, then these $n$ disjoint balls are all contained in
		\[
		B(\bm x_j,r_n+\delta_X/2).
		\]
		Comparing volumes gives
		\[
		n\cdot \frac{4\pi}{3}\left(\frac{\delta_X}{2}\right)^3
		\le
		\frac{4\pi}{3}(r_n+\delta_X/2)^3.
		\]
		Thus
		\[
		r_n\geq (n^{1/3}-1)\frac{\delta_X}{2}.
		\]
		Since $r_n\ge \delta_X$ for all $n\ge 1$, we have 
		\[
		r_n\ge \frac13\delta_X n^{1/3}. 
		\]
	\end{proof}

		To proceed, fix $j$ and order the other points by distance as in Lemma~\ref{lem:distance-ordering}.  Then
		\[
		\sum_{k\ne j}\abs{K_\kappa(\bm x_j-\bm x_k)}
		\le
		\frac1\kappa\sum_{n=1}^{s-1}\frac1{r_n}.
		\]
		Using \eqref{eq:rn-lower},
		\[
		\sum_{n=1}^{s-1}\frac1{r_n}
		\le
		3\frac{1}{\delta_X}\sum_{n=1}^{s-1}n^{-1/3}.
		\]
		Since
		\[
		\sum_{n=1}^{s-1}n^{-1/3}
		\le
		\int_0^{s-1} t^{-1/3}\dd t
		=
		\frac{3}{2} (s-1)^{2/3}
		\leq \frac{3}{2} s^{2/3},
		\]
		we obtain
		\[
		\sum_{k\ne j}\abs{K_\kappa(\bm x_j-\bm x_k)}
		\le
		\frac{9}{2}\frac{s^{2/3}}{\kappa\delta_X}.
		\]
		Taking the maximum over $j$ proves the claim.
	\end{proof}
	
	\subsection{Proof of the spherical frame theorem}
	\begin{proof}
		Expanding the squared norm yields
		\begin{align*}
			\frac1{4\pi}\int_{\Sph}
			\left|\sum_{j=1}^s \alpha_j e^{i\kappa\bm\omega\cdot\bm x_j}\right|^2\dd\bm\omega
		=
			\sum_{j,k=1}^s \alpha_j\overline{\alpha_k}
			K_\kappa(\bm x_j-\bm x_k) =
			\bm \alpha^*G_X\bm \alpha.
		\end{align*}
		 By Lemma~\ref{lem:sparse-row-sum},
		\[
		\sigma_X
		\le
		\mu_X. 
		\]
		Gershgorin's theorem implies that all eigenvalues of $G_X$ lie in the interval
		$
		[1-\sigma_X,1+\sigma_X].
		$
		Hence \eqref{eq:eigenvalue-bounds} follows.  Finally, since $
		\bm\alpha^*G_X\bm\alpha$
		is the integral in \eqref{eq:sparse-frame}, the frame inequality follows immediately.
	\end{proof}

	\section{Kernel derivatives and nonlocal energies}\label{sec:energy}
	
	The local geometry of the MUSIC wells is determined by the Taylor expansion of $K_\kappa$ at the origin:
	\[
	K_\kappa(\bm{h})
	=1-\frac{\kappa^2\abs{\bm{h}}^2}{6}
	+\frac{\kappa^4\abs{\bm{h}}^4}{120}
	+O(\kappa^6\abs{\bm{h}}^6).
	\]
	Hence
	\begin{equation}
		\Psi_\kappa:=-\nabla^2K_\kappa(\bm{0})=\frac{\kappa^2}{3}\Id_3.\label{eq:5.1}
	\end{equation}
		\begin{lemma}
		\label{lem:kernel_derivatives_matched}
		For 
		\[
		\abs{\bm{h}}\le \rho:=\frac{\gamma}{\kappa}, 
		\]
	 we have
		\begin{eqnarray}
		\label{eq:5.2}
		\abs{\nabla K_\kappa(\bm{h})}&\le &\frac{1}{3}\kappa^2\abs{\bm{h}},\\
\label{eq:5.3}
	 \norm{\nabla^2K_\kappa(\bm{h})+\Psi_\kappa}
		&\le& \frac{1}{10}\gamma^2 \kappa^2,\\
		\abs{\Delta K_\kappa(\bm{0})}&=&\kappa^2.
		\end{eqnarray}
	\end{lemma}
	
			\begin{proof}
		First, observe that	
		$$
		K_\kappa(\bm{h})
		=
		\frac{1}{4\pi}\int_{S^2}e^{i\kappa\bm{\omega}\cdot\bm{h}}\,d\bm{\omega}
		=
		\frac{1}{4\pi}\int_{S^2}\cos(\kappa\bm{\omega}\cdot\bm{h})\,d\bm{\omega}.
		$$
		
		Fix a unit vector $\bm{v}\in\mathbb{R}^3$, we have
		
		$$
		\partial_{\bm{v}}K_\kappa(\bm{h})
		=
		-\frac{\kappa}{4\pi}\int_{S^2}(\bm{\omega}\cdot\bm{v})\sin(\kappa\bm{\omega}\cdot\bm{h})\,d\bm{\omega}.
		$$
		
		Since $|\sin t|\leq t$, we get by Cauchy-Schwartz inequality that 
		
		$$
		\begin{aligned}
			|\partial_{\bm{v}}K_\kappa(\bm{h})|
			&\leq
			\frac{\kappa^2}{4\pi}\int_{S^2}|\bm{\omega}\cdot\bm{v}|\,|\bm{\omega}\cdot\bm{h}|\,d\bm{\omega}
			\\
			&\leq
			\frac{\kappa^2}{4\pi}
			\left(\int_{S^2}|\bm{\omega}\cdot\bm{v}|^2\,d\bm{\omega}\right)^{1/2}
			\left(\int_{S^2}|\bm{\omega}\cdot\bm{h}|^2\,d\bm{\omega}\right)^{1/2}
			\\
			&=
			\frac{\kappa^2}{4\pi}
			\left(\frac{4\pi|\bm{v}|^2}{3}\right)^{1/2}
			\left(\frac{4\pi|\bm{h}|^2}{3}\right)^{1/2}\leq
			\frac{\kappa^2}{3}|\bm{h}|.
		\end{aligned}
		$$
		
		Thus, $|\nabla K_\kappa(\bm{h})|\leq \frac{1}{3}\kappa^2|\bm{h}|$.
		
		Next, differentiating gives
		
		$$
		\nabla^2K_\kappa(\bm{h})
		=
		-\frac{\kappa^2}{4\pi}\int_{S^2}\cos(\kappa\bm{\omega}\cdot\bm{h})\bm{\omega}\bm{\omega}^{T}\,d\bm{\omega}.
		$$
		
		At $\bm{h}=0$,
		
		$$
		\begin{aligned}
			\nabla^2K_\kappa(0)
			&=
			-\frac{\kappa^2}{4\pi}\int_{S^2}\bm{\omega}\bm{\omega}^{T}\,d\bm{\omega}
			=
			-\frac{\kappa^2}{3}I_3.
		\end{aligned}
		$$
		
		Thus,
		
		$$
		\nabla^2K_\kappa(\bm{h})+\Psi_\kappa
		=
		\frac{\kappa^2}{4\pi}\int_{S^2}\bigl(1-\cos(\kappa\bm{\omega}\cdot\bm{h})\bigr)\bm{\omega}\bm{\omega}^{T}\,d\bm{\omega}.
		$$
		
		Since this matrix is positive semi-definite, we have
		
		$$
		\|\nabla^2K_\kappa(\bm{h})+\Psi_\kappa\|_2
		=
		\sup_{|\bm{v}|=1}\bm{v}^{T}\bigl(\nabla^2K_\kappa(\bm{h})+\Psi_\kappa\bigr)\bm{v}.
		$$
		
		For a unit vector $\bm{v}$,
		
		$$
		\bm{v}^{T}\bigl(\nabla^2K_\kappa(\bm{h})+\Psi_\kappa\bigr)\bm{v}
		=
		\frac{\kappa^2}{4\pi}\int_{S^2}\bigl(1-\cos(\kappa\bm{\omega}\cdot\bm{h})\bigr)(\bm{\omega}\cdot\bm{v})^2\,d\bm{\omega}.
		$$
		
		Since $0\leq 1-\cos t\leq \frac{t^2}{2}$, we have
		
		$$
		\bm{v}^{T}\bigl(\nabla^2K_\kappa(\bm{h})+\Psi_\kappa\bigr)\bm{v}
		\leq
		\frac{\kappa^4}{8\pi}\int_{S^2}(\bm{\omega}\cdot\bm{v})^2(\bm{\omega}\cdot\bm{h})^2\,d\bm{\omega}.
		$$
Applying the Cauchy-Schwartz inequality as above we have  	
		\begin{eqnarray}
		\int_{S^2}(\bm{\omega}\cdot\bm{v})^2(\bm{\omega}\cdot\bm{h})^2\,d\bm{\omega}
		&\le&
			\frac{4\pi \bm h^2}{5},
\end{eqnarray}
and hence 
		$$
		\begin{aligned}
			\|\nabla^2K_\kappa(\bm{h})+\Psi_\kappa\|_2
			&\leq
			\frac{1}{10}\kappa^4|\bm{h}|^2.
		\end{aligned}
		$$
		
		If $|\bm{h}|\leq \gamma/\kappa$ then
		
		$$
		\|\nabla^2K_\kappa(\bm{h})+\Psi_\kappa\|_2
		\leq
		\frac{1}{10}\gamma^2\kappa^2.
		$$
		We also have		
		$$
		\begin{aligned}
			\Delta K_\kappa(\bm{h})
			&=
			-\frac{\kappa^2}{4\pi}\int_{S^2}e^{i\kappa\bm{\omega}\cdot\bm{h}}\,d\bm{\omega}
			=
			-\kappa^2K_\kappa(\bm{h})
		\end{aligned}
		$$
		 and hence $
		|\Delta K_\kappa(0)|=\kappa^2.
		$
	\end{proof}
		
	We now isolate the nonlocal energy terms which control the cumulative influence of all scatterers except the locally closest one.  For each $j=1,\dots, s$, define
	\beq
	E_0&:=&
	\sup_{\bm{z}\in B_1(0)}
	\sum_{j:\,\abs{\bm{z}-\bm{x}_j}\ge \delta_X/2}
	\abs{K_\kappa(\bm{z}-\bm{x}_j)}^2,\\
	E_1&:=&
	\sup_{\bm{z}\in B_1(0)}
	\sum_{j:\,\abs{\bm{z}-\bm{x}_j}\ge \delta_X/2}
	\abs{\nabla K_\kappa(\bm{z}-\bm{x}_j)}^2\\
		E_2
		&:=&
		\sup_{\bm z\in B_1(0)}
		\sum_{j:\,|\bm z-\bm x_j|\geq \delta_X/2}
		|\nabla^2 K_\kappa(\bm z-\bm x_j)|^2 .
		\eeq
\begin{lemma}		\label{lem:nonlocal_energy_matched}	\label{lem-second-derivative-nonlocal-energy-estimate}
If \(\kappa\delta_X\ge2\), then following estimates hold:		\[
		E_0\le \frac{\nu_X}{c_\nu^2},
		\qquad
		{E_1}\le\nu_X \kappa^2,\qquad	
		E_2
		\leq
		\frac{4}{c_\nu^2}\nu_X \kappa^4. 
		\]
	\end{lemma}
	
	\begin{proof}
  Fix $\bm{z}\in B_1(0)$ and order the points satisfying $\abs{\bm{z}-\bm{x}_j}\ge \delta_X/2$ by increasing distance from $\bm{z}$:
		\[
		r_1\le r_2\le \cdots \le r_N,
		\qquad r_n=\abs{\bm{z}-\bm{x}_{j_n}},
		\qquad N\le s.
		\]
		Since the points of $X$ are $\delta_X$-separated, the same packing argument as in \cref{lem:distance-ordering} gives
		\[
		r_n\ge c\delta_X n^{1/3},
		\qquad n=1,\ldots,N, \quad c= \frac{1}{4}.
		\]
		The trivial bound
		\[
			|K_\kappa(\bm{h})|=\left|
		\frac{\sin(\kappa\abs{\bm{h}})}{\kappa\abs{\bm{h}}}\right|
		\le \frac{1}{\kappa\abs{\bm{h}}}
		\]
		therefore gives
		\[        \sum_{j:\,\abs{\bm{z}-\bm{x}_j}\ge \delta_X/2}
		\abs{K_\kappa(\bm{z}-\bm{x}_j)}^2
		\le
		\frac{1}{\kappa^2}\sum_{n=1}^{N}\frac1{r_n^2} 
		\le
		\frac{1}{c^2\kappa^2\delta_X^2}\sum_{n=1}^{s} n^{-2/3}
		\le
		48\frac{s^{1/3}}{\kappa^2\delta_X^2}.
		\]
		Thus, taking the supremum over $\bm{z}$ proves the estimate for $E_0$.
		
		For the gradient, differentiating
		\[
		K_\kappa(\bm h)
		=
		f(\kappa |\bh|),
		\qquad
		f(t)=\frac{\sin t}{t}
		\]
		yields the pointwise bound
		\[
		\abs{\nabla K_\kappa(\bm{h})}
		\le \left| \frac{\kappa |\bm{h}|\cos(\kappa|\bm{h}|)-\sin(\kappa|\bm{h}|)}{\kappa |\bm{h}|} \right|\frac{1}{\abs{\bm{h}}} \leq c_\nu\frac{1}{\abs{\bm{h}}} ,
		\qquad \bm{h}\ne \bm{0},
		\]
		where $c_\nu \approx 1.0631$ is the global maximum of the function $t\mapsto \abs{\frac{t\cos(t)-\sin(t)}{t}}.$
		Thus the same ordered-distance estimate gives
		\begin{align*}
			\sum_{j:\,\abs{\bm{z}-\bm{x}_j}\ge \delta_X/2}
			\abs{\nabla K_\kappa(\bm{z}-\bm{x}_j)}^2
			&\le
			c_\nu^2\sum_{n=1}^{N}\frac1{r_n^2} \le
			\frac{c_\nu^2}{c^2\delta_X^2}\sum_{n=1}^{s}n^{-2/3}
			\le
			\frac{3c_\nu^2}{c^2}\frac{s^{1/3}}{\delta_X^2}=48c_\nu^2\frac{s^{1/3}}{\delta_X^2}.
		\end{align*}
		Taking the supremum proves the estimate for $E_1$ with  $C_{\rm E} =48c_\nu^2$. 

	Now let us turn to $E_2$.	We have
		\[
		\nabla^2K_\kappa(\bm h)
		=
		\kappa^2 f''(\kappa |\bh|)\bm e\bm e^T
		+
		\kappa^2
		\frac{f'(\kappa |\bh|)}{\kappa |\bh|}
		\left(I-\bm e\bm e^T\right),\quad	\bm e=\frac{\bm h}{|\bm h|}.
				\]
				
The function $f$ satisfies the differential equation $$f''(t)+\frac{2}{t}f'(t)+f(t)=0.$$	
		Since $\abs{f'(t)}\leq \frac{1}{2}$ and \(|f(t)|\le1/t\), 
		we have
		$$\abs{f''(t)}\leq \frac{2}{t}, \qquad \abs{\frac{f'(t)}{t}}\leq \frac{1}{2t}$$
		and hence
		$$\norm{\nabla^2K_\kappa(\bm{h})} \leq \frac{2\kappa}{\abs{\bm{h}}}.$$
		
		Fix \(\bm z\in B_1(0)\), and order the points satisfying
		\[
		|\bm z-\bm x_j|\geq \frac{\delta_X}{2}
		\]
		by increasing distance from \(\bm z\):
		\[
		r_1\leq r_2\leq \cdots \leq r_N,
		\qquad
		r_n=|\bm z-\bm x_{j_n}|.
		\]
		Since in the definition of \(E_2\) we only consider
		\[
		r_j\geq \frac{\delta_X}{2},
		\]
		and since \(\kappa\delta_X\geq 2\), we indeed have \(\kappa r_j\geq 1\).
		
		Therefore,
		\[
		\|\nabla^2K_\kappa(\bm{z}-\bm{x_{j_k}})\|_2^2
		\leq
		\frac{4\kappa^2}{r_k^2}
		\]

		By the packing estimate,
		\[
		r_n\geq \frac{1}{4}\delta_X n^{1/3}.
		\]
		Thus
		\[
		\begin{aligned}
			\sum_{n=1}^N \frac{1}{r_n^2}
			&\leq
			\frac{4^2}{\delta_X^2}
			\sum_{n=1}^s n^{-2/3}
			\leq
			3\cdot 4^2
			\frac{s^{1/3}}{\delta_X^2},
		\end{aligned}
		\]
		and so
		$$
		\sum_{k:\abs{\bm{z}-\bm{x_{j_k}}}\geq{\delta_X}/{2}}\abs{\nabla^2K_\kappa(\bm{z}-\bm{x_{j_k}})}^2\leq\frac{3\cdot 4^3\kappa^2s^{1/3}}{\delta_X^2}		\leq
		\frac{4}{c_\nu^2}\nu_X\kappa^4.
		$$
	\end{proof}
	
	\section{Proof of main  theorems}
	\label{sec:proof}
	
		\subsection{ Local Hessian estimate: proof of \cref{thm:certified-music-landscape} (i)}\label{sec:hessian}
		
		In this section, we establish the local curvature bounds showing that
		the certified neighborhoods are strongly convex MUSIC wells; these
		bounds later imply convergence of fixed-step
		gradient descent. 
		\begin{lemma}
		\label{lem:hessian} Assume \(\rho\leq\delta_X/2,\,\,\kappa\delta_X\ge 2,\,\,\mu_X<1\), \(\gamma=\kappa\rho\).
		
	Then for every $j=1,\dots,s$	and  for all $\bm{z}\in \overline{B_\rho(\bm{x}_j)}$
	 	\beq\label{eq:7.1}
			\norm{\nabla^2 q(\bm{z})-2\Psi_\kappa}&\le &C_2 \kappa^2,\\
			C_2&=& \frac{8}{15} \gamma^2 +6\nu_X+ \frac{2\mu_X}{1-\mu_X} C_1,\\
				C_1&=&
			\frac{\gamma^2}{9}
			+
			\nu_X+\sqrt{1+\nu_X}
			\sqrt{				\left(
				\frac{1}{3}+\frac{\gamma^2}{10}
				\right)^2
				+
				4\nu_X
			}
			\eeq
	and hence
		\beq
		\label{eq:7.20}
		\nabla^2 q(\bm{z})\ge  \left(\frac{2}{3}-C_2\right)\kappa^2\Id_3, 
		\qquad
		\forall\bm{z}\in B_\rho(\bm{x}_j)
		\eeq
	provided $C_2<2/3$.
		Let
		\beq
		C^-_0&:=&	\frac23
			-
			C_2
			-
			\left(
			\frac{2}{\sqrt5}+\frac23\right) 	\eps_{\rm sub}\\
				C^+_0&:=&	\frac23
			+
			C_2
			+
			\left(
			\frac{2}{\sqrt5}+\frac23\right) 	\eps_{\rm sub}.
\eeq
If additionally $C_0^->0$ 
 then					\beq
		\label{eq:7.21}
	C^-_0\kappa^2I_3\le		\nabla^2\widetilde q(\bm{z})
			\le C_0^+\kappa^2I_3.	
		\eeq

		\end{lemma}

	\begin{proof}
			For each fixed $j$, we decompose $\bm{k}(\bm{z})$ into its local and nonlocal parts: 		\[
		\bm{k}(\bm{z})
		=
		K_\kappa(\bm{z}-\bm{x}_j)\bm{e}_j+\sum_{m\neq j}K_\kappa(\bm{z}-\bm{x}_m)\bm e_m
		\]
			where $\bm{e}_j$ is the $j$-th standard basis vector. 		
		Because $\rho\leq\delta_X/2$, for every $m\neq j$ we have
		$$
		|\bm{z}-\bm{x}_m|
		\geq
		|\bm{x}_j-\bm{x}_m|-|\bm{z}-\bm{x}_j|
		\geq
		\frac{\delta_X}{2}
		$$
and by  \cref{lem:nonlocal_energy_matched} that for  $		\bm{r}(\bm{z})
		=
		\big[(1-\delta_{jm})K_\kappa(\bm{z}-\bm{x}_m)\big]_{m=1}^s $ and every
		unit vector $\bm{v}\in\mathbb R^3$,
	\begin{eqnarray*}		\|\bm{r}(\bm{z})\|_2^2
		&\leq&\nu_X
		\\
			\|\partial_{\bm{v}}\bm{r}(\bm{z})\|_2^2
		&\leq&
		\nu_X\kappa^2\\
		\|\partial_{\bm{v}}^2\bm{r}(\bm{z})\|_2^2
		&\leq&
		4\nu_X\kappa^4. 
\end{eqnarray*}

		The eigenvalues of $G_X$ lie in $[1-\mu_X,1+\mu_X]$. Hence
		$$
		G_X^{-1}=I+R_X,
		$$
		where $R$ is Hermitian and
		$$
		\|R_X\|_2
		\leq
		\frac{\mu_X}{1-\mu_X}.
		$$
		With this let us decompose the MUSIC function as 
		$$
		q=f_1-f_2-f_3,
		$$
		where
	\begin{eqnarray*}
		f_1(\bm{z})
		&:=&
		1-|K_\kappa(\bm{z}-\bm{x}_j)|^2\\
		f_2(\bm{z})
		&:=&
		\|\bm{r}(\bm{z})\|_2^2\\
		f_3(\bm{z})
		&:=&
		\bm{k}(\bm{z})^\ast R_X\bm{k}(\bm{z}).
		\end{eqnarray*}
		We estimate the Hessians of these three terms separately.
		
		\medskip
		
		\noindent\textbf{Step 1: The local term.}
		
		Since $K_\kappa$ is real-valued,
		$$
		\nabla^2f_1(\bm{z})
		=
		-2K_\kappa(\bm{h})\nabla^2K_\kappa(\bm{h})
		-
		2\nabla K_\kappa(\bm{h})
		\nabla K_\kappa(\bm{h})^T,\quad \bh:=\bz-\bx_j. 
		$$
		From \eqref{eq:5.1} we have		$$
		\nabla^2K_\kappa(\bm{h})
		=
		-\Psi_\kappa+R(\bm{h}).
		$$
	By Lemma \ref{lem:kernel_derivatives_matched}	we have 
\begin{eqnarray*}
		|R(\bm{h})|
		&\leq&
		\frac{1}{10}\kappa^4|\bm{h}|^2
		\leq
		\frac{\gamma^2}{10}\kappa^2\\
				|\nabla K_\kappa(\bm{h})|
		&\leq&
		\frac{1}{3}\kappa^2|\bm{h}|
		\leq
		\frac{\gamma}{3}\kappa\\
		1-K_\kappa(\bm{h})
		&\leq&
		\frac{\kappa^2|\bm{h}|^2}{6}
		\leq
		\frac{\gamma^2}{6}
	\end{eqnarray*}
	and hence 		$$
		\begin{aligned}
			\left\|
			\nabla^2f_1(\bm{z})-2\Psi_\kappa
			\right\|_2
			\leq{}&
			2|1-K_\kappa(\bm{h})|
			\|\Psi_\kappa\|_2
			+
			2|K_\kappa(\bm{h})|
			\|R(\bm{h})\|_2+
			2|\nabla K_\kappa(\bm{h})|^2.
		\end{aligned}
		$$
		Using $|K_\kappa(\bm{h})|\leq1$, we obtain
		$$
		\begin{aligned}
			\left\|
			\nabla^2f_1(\bm{z})-2\Psi_\kappa
			\right\|_2
			&\leq
			\left(
			\frac{1}{9}
			+
			\frac{1}{5}
			+
			\frac{2}{9}
			\right)
			\gamma^2\kappa^2=
			\frac{8}{15}\gamma^2\kappa^2.
		\end{aligned}
		$$
			
		\medskip
		
		\noindent\textbf{Step 2: The nonlocal Euclidean term.}
		
		Differentiating twice gives
		$$
		\partial_{\bm{v}}^2f_2(\bm{z})
		=
		2\|\partial_{\bm{v}}\bm{r}(\bm{z})\|_2^2
		+
		2\operatorname{Re}
		\left\langle
		\partial_{\bm{v}}^2\bm{r}(\bm{z}),
		\bm{r}(\bm{z})
		\right\rangle
		$$
		for any unit vector $\bm{v}\in\mathbb R^3$. 
		Therefore
		$$
		\left|
		\partial_{\bm{v}}^2f_2(\bm{z})
		\right|
		\leq
		2\|\partial_{\bm{v}}\bm{r}(\bm{z})\|_2^2
		+
		2\|\partial_{\bm{v}}^2\bm{r}(\bm{z})\|_2
		\|\bm{r}(\bm{z})\|_2.
		$$
		\cref{lem:nonlocal_energy_matched} implies
		$$
		\|\bm{r}(\bm{z})\|_2
		\leq
		\sqrt{\nu_X},
		\qquad
		\|\partial_{\bm{v}}^2\bm{r}(\bm{z})\|_2
		\leq
		2\sqrt{\nu_X}\,\kappa^2.
		$$
		Hence
		$$
		\left|
		\partial_{\bm{v}}^2f_2(\bm{z})
		\right|
		\leq
		2\nu_X\kappa^2+4\nu_X\kappa^2
		=
		6\nu_X\kappa^2.
		$$
		Taking the supremum over all unit vectors $\bm{v}$ gives
		$$
		\|\nabla^2f_2(\bm{z})\|_2
		\leq
		6\nu_X\kappa^2.
		$$
		
		\medskip
		
		\noindent\textbf{Step 3: The Gram perturbation term.}
		
		For every unit vector $\bm{v}\in\mathbb R^3$,
		$$
		\begin{aligned}
			\partial_{\bm{v}}^2f_3(\bm{z})
			={}&
			2\operatorname{Re}
			\left\langle
			\partial_{\bm{v}}^2\bm{k}(\bm{z}),
			R_X\bm{k}(\bm{z})
			\right\rangle+
			2
			\left\langle
			\partial_{\bm{v}}\bm{k}(\bm{z}),
			R_X\partial_{\bm{v}}\bm{k}(\bm{z})
			\right\rangle
		\end{aligned}
		$$
	and hence		$$
		\left|
		\partial_{\bm{v}}^2f_3(\bm{z})
		\right|
		\leq
		2\|R_X\|_2
		\left(
		\|\partial_{\bm{v}}^2\bm{k}(\bm{z})\|_2
		\|\bm{k}(\bm{z})\|_2
		+
		\|\partial_{\bm{v}}\bm{k}(\bm{z})\|_2^2
		\right).
		$$

First, write
\[
\bm{k}(\bm{z})
=
K_\kappa(\bz-\bx_j)\bm e_j+\bm r_j(z),
\] where \(\bm r_j\) has zero \(j\)-th coordinate. Then
\[
\|\bm{k}(\bm{z})\|_2^2
=
|K_\kappa(\bz-\bx_j)|^2+\|\bm r_j(\bz)\|_2^2
\le
1+\frac{\nu_X}{c_\nu^2}
\le1+\nu_X.
\] 		
		Second, \eqref{eq:5.2} and \cref{lem:nonlocal_energy_matched} imply
		\beqn
		\|\partial_{\bm{v}}\bm{k}(\bm{z})\|_2^2
		&\leq&
		\left(
		\frac{\gamma^2}{9}+\nu_X
		\right)\kappa^2.
		\eeqn
		Third, \eqref{eq:5.3} and  \cref{lem:nonlocal_energy_matched} imply
				$$
		\|\partial_{\bm{v}}^2\bm{k}(\bm{z})\|_2^2
		\leq
		\left[
		\left(
		\frac{1}{3}+\frac{\gamma^2}{10}
		\right)^2
		+
		4\nu_X
		\right]\kappa^4.
		$$
		Combining these estimates gives
		$$
		\begin{aligned}
			\left|
			\partial_{\bm{v}}^2f_3(\bm{z})
			\right|
			&\leq&
			\frac{2\mu_X}{1-\mu_X}C_1\kappa^2,\qquad
			C_1=
			\Bigg[
			\frac{\gamma^2}{9}
			+
			\nu_X+\sqrt{1+\nu_X}
			\sqrt{				\left(
				\frac{1}{3}+\frac{\gamma^2}{10}
				\right)^2
				+
				4\nu_X
			}
			\Bigg].
		\end{aligned}
		$$
			Taking the supremum over all unit vectors $\bm{v}$ yields
		$$
		\|\nabla^2 f_3(\bm{z})\|_2
		\leq
		\frac{2\mu_X}{1-\mu_X}C_1\kappa^2. 
		$$
	Combining the three estimates, we  have \eqref{eq:7.20} and by \eqref{eq:8.1}, \eqref{eq:7.21}. 
		
	\end{proof}

		\subsection{ Upper threshold: proof of \cref{thm:certified-music-landscape} (ii).}\label{sec:threshold}
	In this section, we establish  \cref{thm:certified-music-landscape} (ii)
	using the spherical frame bounds \cref{thm:sparse-spherical-frame} and nonlocal energy estimates \cref{lem:nonlocal_energy_matched}.
	
		\begin{lemma}[Principal-angle stability]
\label{lem:principal-angle-stability}
For a unit vector \(u\in\mathcal H\), define
$$
    \alpha_{\mathcal U}(u)
    :=
     \arccos\|P_\cU u\|_2,\qquad  \alpha_{\widetilde \cU}(u)
    :=
     \arccos\|P_{\widetilde\cU} u\|_2. 
$$
 Then
\begin{equation}
\label{eq:subspace-angle-lipschitz}
    \left|
        \alpha_{\mathcal U}(u)
        -
        \alpha_{\widetilde{\mathcal U}}(u)
    \right|
    \leq
    \arcsin\epssub.
\end{equation}
Consequently,
\begin{equation}
\label{eq:principal-angle-residual}
    \|(I-P_{\widetilde \cU})u \|_2
    \geq
    \sin\left(
        \left(
            \arcsin\|(I-P_\cU)u\|_2
            -
            \arcsin\epssub
        \right)_+
    \right).
\end{equation}
\end{lemma}

\begin{proof}
Set
\[
    \theta:=\arcsin\epssub. 
\]
For equal-dimensional finite-dimensional subspaces, \(\theta\) is
their largest principal angle.
We first prove
\[
    \alpha_{\mathcal U}(u)
    \leq
    \alpha_{\widetilde{\mathcal U}}(u)+\theta.
\]
If \(P_{\widetilde U} u=0\), then
\(\alpha_{\widetilde{\mathcal U}}(u)=\pi/2\), and the claim is
immediate.  Otherwise, let
\[
    v:=\frac{P_{\widetilde U}u}{\|P_{\widetilde U}u\|_2}
    \in\widetilde{\mathcal U}.
\]
Then
\[
    \angle(u,v)
    =
    \alpha_{\widetilde{\mathcal U}}(u).
\]
Since \(v\in\widetilde{\mathcal U}\),
\[
    \|(I-P_\cU)v\|_2
    =
    \|(P_{\widetilde \cU}-P_\cU)v\|_2
    \leq\epssub
\]
and hence
\[
    \angle(v,\mathcal U)
    =
    \arcsin\|(I-P_\cU)v\|_2
    \leq\theta.
\]
The triangle inequality for acute angles gives
\[
    \alpha_{\mathcal U}(u)
    \leq
    \angle(u,v)+\angle(v,\mathcal U)
    \leq
    \alpha_{\widetilde{\mathcal U}}(u)+\theta.
\]
Interchanging \(\mathcal U\) and
\(\widetilde{\mathcal U}\) proves
\eqref{eq:subspace-angle-lipschitz}.

It follows that
\[
    \alpha_{\widetilde{\mathcal U}}(u)
    \geq
    \left(
        \alpha_{\mathcal U}(u)-\theta
    \right)_+.
\]
Taking sine, which is increasing on \([0,\pi/2]\), proves
\eqref{eq:principal-angle-residual}.
\end{proof}

		\begin{lemma}
		\label{lem:outside_lower_matched} 
		Suppose $\mu_X<1$ and $A_\gamma$ as given in \eqref{eq:A-gamma}  is positive. 
Then 	
			\beq
			\label{eq:8.11} q(\bz)\geq A_\gamma,\quad \forall \bz\in B_1(0)\setminus\bigcup_{\ell=1}^s B_\rho(\bm{x}_\ell)
			\eeq
			and 
			\begin{equation}
\label{eq:angular-outside-lower}
    \widetilde q(\bz)
    \geq
    \sin^2\left(
        \left(
            \arcsin\sqrt{A_\gamma}
            -
            \arcsin\epssub
        \right)_+
    \right), \quad \forall \bz\in B_1(0)\setminus\bigcup_{\ell=1}^s B_\rho(\bm{x}_\ell). 
\end{equation}
		\end{lemma}

	\begin{proof} By the frame bounds, we have 
\[
(1+\mu_X)^{-1}\le G_X^{-1}\le  (1-\mu_X)^{-1}
\] and hence
\[
        \bm{k}(\bm{z})^*G_X^{-1}\bm{k}(\bm{z})
        \le (1-\mu_X)^{-1}\abs{\bm{k}(\bm{z})}_2^2.
\]

From \eqref{eq:kernel-vector} we have
		\beq
		\label{eq:7.11}
		\norm{\bm{k}(\bm{z})}^2
		=
		\sum_{j}\abs{K_\kappa(\bm{z}-\bm{x}_j)}^2. 
		\eeq
	If $\rho \ge \delta_X/2$ then 
		\beq
		\label{eq:7.12}
		\norm{\bm{k}(\bm{z})}^2\le E_0\le\nu_X;
		\eeq
if, however, $\rho<\delta_X/2$, then 
\beq
\nonumber
        \norm{\bm{k}(\bm{z})}^2
       & \le&
        \sup_{\abs{\bm{h}}\ge\rho}\abs{K_\kappa(\bm{h})}^2
        +
        \sum_{j\ne j_0}\abs{K_\kappa(\bm{z}-\bm{x}_j)}^2\\
        &\le& \sup_{t\ge \gamma} \sinc^2(t)+E_0\nonumber\\
        & \le & 1-b_\gamma+\nu_X\label{eq:7.13}
\eeq
where  $j_0$ is an index of a nearest point to $\bm{z}$.  
Indeed, if $|\bz-\bx_{j_0}|\geq\delta_X/2$, every term is included
in the definition of $E_0$; otherwise, for $j\ne j_0$,
\[
    |\bz-\bx_j|
    \geq|\bx_j-\bx_{j_0}|-|\bz-\bx_{j_0}|
    >\frac{\delta_X}{2}.
\]
 In either case, 
		\beq
		\label{eq:7.14}
		\bm{k}(\bm{z})^*G_X^{-1}\bm{k}(\bm{z})
		\le (1-\mu_X)^{-1}(1-b_\gamma+\nu_X)
		\eeq
		and hence \eqref{eq:8.11}. 
Apply \cref{lem:principal-angle-stability} to the unit Fourier atom
\[
    u=\varphi_{\bz}.
\]
Then \eqref{eq:8.11} implies \eqref{eq:angular-outside-lower}. 

\end{proof}

\begin{corollary}[Sharp threshold condition]
\label{cor:sharp-angular-threshold}
Suppose
\[
    0<A_\gamma\leq1
\]
and that
\[
    q(\bz)\geq A_\gamma,\qquad \forall \bz\in B_1(0)\setminus\bigcup_{\ell=1}^s B_\rho(\bm{x}_\ell).
\]
Set
\[
    \alpha_\gamma:=\arcsin\sqrt{A_\gamma},
    \qquad
    \theta:=\arcsin\epssub.
\]
Let \[
    \tau_0:=\epssub^2
\]
and
\[
    \tau_1
    :=
    \sin^2\left(
        \left(
            \alpha_\gamma-\theta
        \right)_+
    \right).
\]
A strict threshold gap
\[
    \tau_0<\tau_1
\]
holds if and only if
\begin{equation}
\label{eq:sharp-angular-condition}
\epssub
    <
    \sin\left(
        \frac12\arcsin\sqrt{A_\gamma}
    \right).
\end{equation}
Equivalently,
\begin{equation}
\label{eq:sharp-polynomial-condition2} 
   \epssub < \frac{\sqrt{1-\sqrt{1-A_\gamma}}}{\sqrt 2} .
\end{equation}

\end{corollary}

\begin{proof}
Note that  \(\tau_0=\sin^2\theta\) and the right hand side of 
\eqref{eq:angular-outside-lower} equals 
\(\tau_1\).

If \(\alpha_\gamma\leq\theta\), then \(\tau_1=0\), so a strict gap is
impossible.  If \(\alpha_\gamma>\theta\), then all relevant angles
belong to \([0,\pi/2]\), and therefore
\begin{align*}
    \tau_0<\tau_1
    &\iff
    \sin^2\theta
    <
    \sin^2(\alpha_\gamma-\theta)\\
    &\iff
    \theta<\alpha_\gamma-\theta\\
    &\iff
    2\theta<\alpha_\gamma.
\end{align*}
This is equivalent to \eqref{eq:sharp-angular-condition}.

Condition \eqref{eq:sharp-angular-condition} necessarily implies
\(\theta<\pi/4\), or \(\epssub<1/\sqrt2\) and hence the
double-angle identity gives
\begin{align*}
    2\theta<\alpha_\gamma
    &\iff
    \sin(2\theta)<\sin\alpha_\gamma\\
    &\iff
    2\epssub\sqrt{1-\epssub^2}<\sqrt{A_\gamma}\\
    &\iff
    A_\gamma>4\epssub^2(1-\epssub^2),
\end{align*}
which proves \eqref{eq:sharp-polynomial-condition2}.
\end{proof}

\begin{remark}
\label{rem:angular-low-noise-branch}
In the present MUSIC theorem the low-noise condition is automatic.  Once
\(\mu_X<1\), one has \(C_2\geq0\), and hence
\[
    C_0^-
    =
    \frac23-C_2
    -
    \left(
        \frac2{\sqrt5}+\frac23
    \right)\epssub
    >0
\]
implies
\[
 \epssub
    <
    \frac{2}
         {2+6/\sqrt5}
    <\frac1{\sqrt2}.
\]
\end{remark}

	\subsection{Localization error bound: proof of \cref{thm:certified-music-landscape} (iii).}\label{sec:perturbed-minima}
	
We state some basic perturbation estimates about derivatives. 	
	\begin{lemma}
		\label{lem:perturbation}
		For all $\bm{z}\in B_1(0)$, we have
		\beq
				\label{eq:8.2}
			|\nabla \widetilde q(\bm z)-\nabla q(\bm z)|
			&\leq&\frac{2}{\sqrt{3}}\kappa\epssub\\
		\label{eq:8.1}		|\nabla^2 \widetilde q(\bm z)-\nabla^2 q(\bm z)|
			&\leq&
			\left(
			\frac{2}{\sqrt 5}
			+
			\frac23
			\right)
			\kappa^2\eps_{\rm sub}.
		\eeq
	\end{lemma}
	\begin{proof}		
	Let $\bm v\in \mathbb R^3$ be a unit vector. We  have
		\[
		\partial_{\bm v}\varphi_{\bm z}(\bom)
		=
		i\kappa(\bom\cdot \bm v)\varphi_{\bm z}(\bom).
		\]
		Therefore
		\[
		\begin{aligned}
			\|\partial_{\bm v}\varphi_{\bm z}\|_2^2
			&=
			\kappa^2
			\frac{1}{4\pi}
			\int_{S^2}
			(\bom\cdot \bm v)^2\,d\bom.
		\end{aligned}
		\]
		By rotational symmetry,
		\[
		\frac{1}{4\pi}
		\int_{S^2}
		(\bom\cdot \bm v)^2\,d\bom
		=
		\frac13.
		\]
		Hence
		\[
		\|\partial_{\bm v}\varphi_{\bm z}\|_2
		=
		\frac{\kappa}{\sqrt 3}.
		\]
		
		Differentiating
		in the direction $\bm v$, we obtain
		\[
		\partial_{\bm v}(\widetilde q-q)(\bm z)
		=
		2\operatorname{Re}
		\langle
		\partial_{\bm v}\varphi_{\bm z},
		(P_\cU-P_{\widetilde \cU})\varphi_{\bm z}
		\rangle
		\]
and hence
		\[
		\begin{aligned}
			|\partial_{\bm v}(\widetilde q-q)(\bm z)|
			&\leq
			2
			\|\partial_{\bm v}\varphi_{\bm z}\|_2
			\|P_{\widetilde \cU}-P_\cU\|_2
			\|\varphi_{\bm z}\|_2 =
			\frac{2}{\sqrt 3}
			\kappa
			\eps_{\rm sub}.
		\end{aligned}
		\]
		Taking the supremum over all unit vectors \(\bm v\) gives \eqref{eq:8.2}. 
		
		In a similar way we estimate the Hessian. We have
		\[
		\partial_{\bm v}^2\varphi_{\bm z}(\bom)
		=
		-\kappa^2(\bom\cdot \bm v)^2\varphi_{\bm z}(\bom).
		\]
		Therefore
		\[
		\begin{aligned}
			\|\partial_{\bm v}^2\varphi_{\bm z}\|_2^2
			&=
			\frac{\kappa^4}{4\pi}
			\int_{S^2}
			(\bom\cdot \bm v)^4\,d\bom.
		\end{aligned}
		\]
		By rotational symmetry,
		\[
		\frac{1}{4\pi}
		\int_{S^2}
		(\bom\cdot \bm v)^4\,d\bom
		=
		\frac15.
		\]
		Thus
		\[
		\|\partial_{\bm v}^2\varphi_{\bm z}\|_2
		=
		\frac{\kappa^2}{\sqrt 5}.
		\]
		
		Differentiating twice in the direction $\bm v$, we get
		\[
		\begin{aligned}
			\partial_{\bm v}^2(\widetilde q-q)(\bm z)
			&=
			2\operatorname{Re}
			\langle
			\partial_{\bm v}^2\varphi_{\bm z},
			(P_{\widetilde \cU}-P_\cU)\varphi_{\bm z}
			\rangle
			+
			2
			\langle
			\partial_{\bm v}\varphi_{\bm z},
			(P_{\widetilde \cU}-P_\cU)\partial_{\bm v}\varphi_{\bm z}
			\rangle.
		\end{aligned}
		\]
		Hence
		\[
		\begin{aligned}
			|\partial_{\bm v}^2(\widetilde q-q)(\bm z)|
			&\leq
			2
			\|\partial_{\bm v}^2\varphi_{\bm z}\|_2
			\|P_{\widetilde \cU}-P_\cU\|_2
			\|\varphi_{\bm z}\|_2
			+
			2
			\|\partial_{\bm v}\varphi_{\bm z}\|_2^2
			\|P_{\widetilde \cU}-P_\cU\|_2 \\
			&=
			2
			\frac{\kappa^2}{\sqrt 5}
			\eps_{\rm sub}
			+
			2
			\frac{\kappa^2}{3}
			\eps_{\rm sub}.
		\end{aligned}
		\]
		Thus
		\[
		|\partial_{\bm v}^2(\widetilde q-q)(\bm z)|
		\leq
		\left(
		\frac{2}{\sqrt 5}
		+
		\frac23
		\right)
		\kappa^2\eps_{\rm sub}.
		\]
		Taking the supremum over all unit vectors \(\bm v\) gives
\eqref{eq:8.1}. 			\end{proof}
	\color{black}

	Next  we establish the localization error bound \eqref{eq:localization-error} from the local Hessian estimate \cref{lem:hessian} and the perturbation bounds \cref{lem:perturbation}. 
			\begin{lemma}
		\label{lem:local_minima_matched}
Let $C^-_0>0$ be the convexity lower bound in \eqref{eq:7.21}. If 		\beq \label{eq:9.1}
		\epssub
		<\frac{\sqrt{3}}{2} C^-_0\gamma
		\eeq
then $\widetilde q$ has a unique critical point, equivalently the unique minimizer,   $\widetilde\bx_j$ in \(\overline{B_\rho(\bx_j)}\) 		 and
\[
    |\widetilde\bx_j-\bx_j|
    \leq
    \frac{2}{\sqrt3C_0^-}
    \frac{\eps_{\rm sub}}{\kappa}
    <\rho,\quad j=1,\dots,s.
\]
Moreover,
		\[
		\widetilde q(\widetilde{\bm{x}}_j)
		\le \widetilde q(\bm{x}_j)
		\le \epssub^2,\quad j=1,\dots, s.
		\]
	\end{lemma}
	
	\begin{proof}
		Fix $j\in\{1,\ldots,s\}$. Since  $\nabla q(\bm{x}_j)=0$,  Lemma~\ref{lem:perturbation} implies
		\beq
		|{\nabla\widetilde q(\bm{x}_j)}|
		\le \frac{2}{\sqrt{3}}\kappa\epssub.\label{eq:9.10}
		\eeq

		By Lemma~\ref{lem:hessian},
	$$			\nabla^2\widetilde q(\bm{z})
			\ge
		C^-_0\kappa^2I_3,		\qquad
	\forall 
		\bm{z}\in B_\rho(\bm{x}_j).
		$$

		We first prove the existence of a critical point in $B_\rho(\bm{x}_j)$. 
		Since $\widetilde q$ is continuous, it attains a minimum on the
		compact set
		$
		\overline{B_\rho(\bm{x}_j)}.
		$
		We claim that this minimum cannot occur on the boundary.
		
		Let
		$
		\bm{z}\in\partial B_\rho(\bm{x}_j).
		$
		By the fundamental theorem of calculus,
		$$
		\begin{aligned}
			\nabla\widetilde q(\bm{z})
			-
			\nabla\widetilde q(\bm{x}_j)
			=
			\int_0^1
			\nabla^2\widetilde q
			\bigl(
			\bm{x}_j+t(\bm{z}-\bm{x}_j)
			\bigr)
			(\bm{z}-\bm{x}_j)
			\,dt.
		\end{aligned}
		$$
		Taking the inner product with $\bm{z}-\bm{x}_j$ and using the
		Hessian lower bound gives
		$$
		\begin{aligned}
			\nabla\widetilde q(\bm{z})
			\cdot
			(\bm{z}-\bm{x}_j)
			\geq&
			-|\nabla\widetilde q(\bm{x}_j)|
			|\bm{z}-\bm{x}_j|+C^-_0\kappa^2
			|\bm{z}-\bm{x}_j|^2.
		\end{aligned}
		$$
		Since
		$$
		|\bm{z}-\bm{x}_j|
		\leq
		\rho
		=
		\frac{\gamma}{\kappa},
		$$
		we obtain
		$$
		\begin{aligned}
			\nabla\widetilde q(\bm{z})
			\cdot
			(\bm{z}-\bm{x}_j)
			&\geq
					-\frac{2\gamma}{\sqrt3}\epssub
			+
			C^-_0{\gamma^2}>0,\quad \forall
		\bm{z}\in\partial B_\rho(\bm{x}_j)
		\end{aligned}
$$		under  the assumption
				\eqref{eq:9.1}. 
		Thus the gradient points strictly outward on the boundary and  a
		minimizer of $\widetilde q$ over
		$\overline{B_\rho(\bm{x}_j)}$ cannot lie on the boundary.
		It follows that the minimizer lies in the interior
		$B_\rho(\bm{x}_j)$ and is therefore a critical point. Strong
		convexity implies that this critical point $
		\widetilde{\bm{x}}_j
		$ is unique. 
		
	Set
		$
		\bm{h}_j
		:=
		\widetilde{\bm{x}}_j-\bm{x}_j.
		$
		We have 
		$$
		\nabla\widetilde q(\widetilde{\bm{x}}_j)
		-
		\nabla\widetilde q(\bm{x}_j)
		=
		\int_0^1
		\nabla^2\widetilde q
		\bigl(
		\bm{x}_j+t\bm{h}_j
		\bigr)
		\bm{h}_j
		\,dt
		$$
	and hence 
		$$
		\begin{aligned}
			\left(
			\nabla\widetilde q(\widetilde{\bm{x}}_j)
			-
			\nabla\widetilde q(\bm{x}_j)
			\right)
			\cdot\bm{h}_j
			&=
			\int_0^1
			\bm{h}_j^T
			\nabla^2\widetilde q
			\bigl(
			\bm{x}_j+t\bm{h}_j
			\bigr)
			\bm{h}_j
			\,dt\geq
			C^-_0{\kappa^2}
			|\bm{h}_j|^2.
		\end{aligned}
		$$
		Since
		$
		\nabla\widetilde q(\widetilde{\bm{x}}_j)
		=
		\bm{0},
		$
		we conclude that
		$$
		|\bm{h}_j|
		\leq
		\frac{1}{C^-_0\kappa^2}
		|\nabla\widetilde q(\bm{x}_j)| \le 	\frac{2}{C^-_0\sqrt3}
			\frac{\epssub}{\kappa}
		$$
	by 	\eqref{eq:9.10}. 	Furthermore, the assumed bound on $\epssub$ gives
		$$
		\begin{aligned}
			|\widetilde{\bm{x}}_j-\bm{x}_j|\leq	\frac{2}{C^-_0\sqrt3}
			\frac{\epssub}{\kappa}
			&<\rho
		\end{aligned}
		$$
i.e. 		$
		\widetilde{\bm{x}}_j
		\in
		B_{\rho}(\bm{x}_j).
		$
		
		It remains to estimate the value of $\widetilde q$ at this minimum.
	To this end, we have the calculation 
			$$
		\begin{aligned}
			\widetilde q(\bm{x}_j)
			=
			\|
			(\Id-\Proj_{\widetilde{\cU}})
			\varphi_{\bm{x}_j}
			\|_2^2=
			\|
			(\Proj_{\cU}-\Proj_{\widetilde{\cU}})
			\varphi_{\bm{x}_j}
			\|_2^2\leq
			\|
			\Proj_{\cU}-\Proj_{\widetilde{\cU}}
			\|_2^2
			\|\varphi_{\bm{x}_j}\|_2^2=
			\epssub^2
		\end{aligned}
		$$
	and hence	$$
			\widetilde q(\widetilde{\bm{x}}_j)
			\leq
			\widetilde q(\bm{x}_j)
			\leq
			\epssub^2
		$$
		since $\widetilde{\bm{x}}_j$ minimizes $\widetilde q$ over
		$\overline{B_\rho(\bm{x}_j)}$. 
	\end{proof}

\subsection{Threshold accessibility: proof of  \cref{thm:certified-music-landscape} (iv).}\label{sec:initialization}

	Modifying the definition in  \cite{gradient-music} we adopt the following
	objective-only version.
	
	\begin{definition}[Threshold-admissible landscape]
		\label{def:admissible}
		Let $f\colon B_1(0)\to [0,\infty)$, let
		\[
		X=\{\bm{x}_1,\dots,\bm{x}_s\}\subset B_1(0),
		\]
		and let $\eps,\rho,\tau_0,\tau_1$ satisfy
		\[
			0\leq\eps<\rho,
			\qquad
			0\le \tau_0<\tau_1.
		\]
		We say that $f$ is a threshold-admissible optimization landscape for $X$,
		with parameters
		$(\eps,\rho,\tau_0,\tau_1)$, if:
		\begin{enumerate}[label=\textup{(\roman*)}]
			\item for every $j=1,\dots,s$, the function $f$ has a 
			local minimum
			$\widetilde{\bm{x}}_j\in
			\overline{B_\eps(\bm{x}_j)}$;
			\item the relevant minima are uniformly shallow and all points
			outside the certified neighborhoods are uniformly higher:
			\beq\label{eq:2.20}
			\max_{1\le j\le s} f(\widetilde{\bm{x}}_j)&\le& \tau_0,\\
			f(\bm{z})&\ge &\tau_1, 
			\quad
			\forall
			\bm{z}\in B_1(0)\setminus
			\bigcup_{j=1}^s B_\rho(\bm{x}_j).\label{eq:2.2}
			\eeq
		\end{enumerate}
	\end{definition}
		A threshold-admissible landscape has threshold-accessible relevant neighborhoods.	
	The next lemma formalizes this observation.	
	\begin{lemma}[Threshold-accessibility]
		\label{lem:threshold}\label{lem:threshold-coverage}
		Suppose that $f\colon B_1(0)\to [0,\infty)$ is an $L$-Lipschitz threshold-admissible  landscape for $\{\bm{x}_\ell\}_{\ell=1}^s$ with parameters $(\eps,\rho,\tau_0,\tau_1)$. Let $G\subset B_1(\bm{0})$ be a finite set satisfying
		\[
		\mesh_{B_1}(G)
	<
		\min\left\{
		\rho-\eps,\,
		\frac{\tau_1-\tau_0}{L}
		\right\}.
		\]
		Define the thresholded set
		\begin{eqnarray}
		G_{\tau_1}
		:=
		\{\bom\in G:f(\bom)<\tau_1\}.\label{eq1}
		\end{eqnarray}
		By \eqref{eq:2.2}, 	
		$$	
		G_{\tau_1}
		\subseteq
		\bigcup_{\ell=1}^s B_\rho(\bm{x}_\ell).
		$$
In addition,
		\[
		B_\rho(\bm{x}_\ell)\cap G_{\tau_1} \neq \emptyset,\quad  \forall \ell\in\{1,\dots,s\}. 
		\]
	\end{lemma}
	\begin{proof}
		Fix $\ell\in\{1,\ldots,s\}$. By definition of threshold-admissible landscape, there exists a relevant
		local minimum $
		\widetilde{\bm{x}}_\ell
		\in \overline{B_\eps(\bm{x}_\ell)} $
		such that $
		f(\widetilde{\bm{x}}_\ell)\leq\tau_0
		$.

		By the definition of mesh norm, there exists $\bm{z}_\ell\in G$ such
		that
		\[
		|\bm{z}_\ell-\widetilde{\bm{x}}_\ell|
		\leq \mesh_{B_1}(G).
		\]
		Hence, by the triangle inequality,
		\[
		\begin{aligned}
			|\bm{z}_\ell-\bm{x}_\ell|
			&\leq
			|\bm{z}_\ell-\widetilde{\bm{x}}_\ell|
			+
			|\widetilde{\bm{x}}_\ell-\bm{x}_\ell|
			\leq\mesh_{B_1}(G)+\eps
			< \rho,
		\end{aligned}
		\]
i.e.  $
		\bm{z}_\ell\in B_\rho(\bm{x}_\ell).$
		
		On the other hand, since $f$ is $L$-Lipschitz,
		\[
		\begin{aligned}
			f(\bm{z}_\ell)
			&\leq
			f(\widetilde{\bm{x}}_\ell)
			+
			L|\bm{z}_\ell-\widetilde{\bm{x}}_\ell|\leq
			\tau_0
			+
			L\mesh_{B_1}(G) < \tau_1.
		\end{aligned}
		\]
		Thus $\bm{z}_\ell\in G_{\tau_1}$ implying  $\bz_\ell\in B_\rho(\bm{x}_\ell)\cap G_{\tau_1} \neq \emptyset$.
		\end{proof}
	
	The mesh condition can be made fully explicit because the MUSIC
objective is globally Lipschitz.

\begin{lemma}[Global Lipschitz bound]
\label{lem:music-global-lipschitz}
For every orthogonal projector \(P\) on \(\mathcal H\), the function
\[
    q_P(\bz):=1-\|P\varphi_\bz\|_2^2
\]
satisfies
\begin{equation}
\label{eq:music-global-lipschitz}
    |q_P(\bz)-q_P(\bw)|
    \leq\frac{2\kappa}{\sqrt3}|\bz-\bw|,
    \qquad \bz,\bw\in B_1(0).
\end{equation}
\end{lemma}

\begin{proof}
For a unit vector \(\bv\in\mathbb R^3\), differentiation gives
\[
    \partial_\bv q_P(\bz)
    =-2\operatorname{Re}
      \langle P\partial_\bv\varphi_\bz,
              P\varphi_\bz\rangle.
\]
Since \(P\) is a contraction, \(\|\varphi_\bz\|_2=1\), and
\[
    \|\partial_\bv\varphi_\bz\|_2
    =\left(
       \frac{\kappa^2}{4\pi}
       \int_{\mathbb S^2}(\bm\omega\cdot\bv)^2\,d\bm\omega
      \right)^{1/2}
    =\frac{\kappa}{\sqrt3},
\]
we have \(|\partial_\bv q_P(\bz)|\leq2\kappa/\sqrt3\).
Integrating along the segment from \(\bw\) to \(\bz\) proves
\eqref{eq:music-global-lipschitz}.
\end{proof}

Consequently, the thresholded set
	\(G_{\tau_1}\) contains at least one point in each certified
	neighborhood \(B_\rho(\bx_\ell)\). 		The point of \cref{lem:threshold} is that the initialization grid need not be refined to the final reconstruction scale \(\eps\).  The mesh size is instead governed by the neighborhood gap \(\rho-\eps\) and by the threshold gap \(\tau_1-\tau_0\), normalized by the Lipschitz scale \(\|\nabla f\|_{L^\infty}\).  This distinction is important in the small-noise regime: the final error parameter \(\eps\) may be much smaller than the spacing needed for a successful initialization.  Thus the thresholding step may be performed on a relatively coarse grid, as long as the grid resolves each certified neighborhood and the value of \(f\) cannot rise from \(\tau_0\) to \(\tau_1\) within one mesh step. There are other applications that can fit into this framework such as structured matrix and  exponential sum approximations, see \cite{structured-approx} for details.

\subsection{Linear convergence: proof of \cref{thm:certified-music-landscape} (v).}
We next give the algorithmic convergence statement associated with
the admissible MUSIC landscape initialized by certified thresholded set.
\begin{theorem}[Linear convergence]
\label{thm:convergence}
Assume all the hypotheses and numerical inequalities in the statement
of \cref{thm:certified-music-landscape}, and define
\[
    m:=C_0^-\kappa^2,
    \qquad
    M:=C_0^+\kappa^2.
\]
For a fixed step size \(h\) satisfying
\begin{equation}
\label{eq:step-size}
    0<h\leq\frac{2}{(C_0^-+C_0^+)\kappa^2}
    =\frac{2}{m+M},
\end{equation}
let
\[
    T_h(\bz):=\bz-h\nabla\widetilde q(\bz).
\]
Then, for every \(j=1,\ldots,s\),
\[
    T_h\left(\overline{B_\rho(\bx_j)}\right)
    \subset B_\rho(\bx_j).
\]
Moreover, \(T_h\) is a contraction on
\(\overline{B_\rho(\bx_j)}\) with contraction factor
\[
    r_h:=1-hm=1-hC_0^-\kappa^2\in[0,1).
\]
Consequently, every initialization
\(\bz^{(0)}\in\overline{B_\rho(\bx_j)}\) generates iterates
\[
    \bz^{(n+1)}=T_h(\bz^{(n)})
\]
that remain in \(B_\rho(\bx_j)\) and converge to the unique local
minimizer \(\widetilde\bx_j\). More precisely,
\begin{equation}
\label{eq:iterate-linear-rate}
    |\bz^{(n)}-\widetilde\bx_j|
    \leq
    r_h^n|\bz^{(0)}-\widetilde\bx_j|,
    \qquad n\geq0,
\end{equation}
and
\begin{equation}
\label{eq:objective-linear-rate}
    0
    \leq
    \widetilde q(\bz^{(n)})-\widetilde q(\widetilde\bx_j)
    \leq
    \frac{M}{2}r_h^{2n}
    |\bz^{(0)}-\widetilde\bx_j|^2.
\end{equation}
\end{theorem}

\label{sec:convergence-proof}
\begin{proof}
Fix \(j\in\{1,\ldots,s\}\), and abbreviate
\[
    \overline B_j:=\overline{B_\rho(\bx_j)}.
\]
By \eqref{eq:main-hessian-bounds},
\begin{equation}
\label{eq:cor-hessian-interval}
    mI_3
    \preceq
    \nabla^2\widetilde q(\bz)
    \preceq
    MI_3,
    \qquad
    \bz\in\overline B_j.
\end{equation}
Let \(Q\) be any real symmetric matrix satisfying
\(mI_3\preceq Q\preceq MI_3\). For
\(0<h\leq2/(m+M)\), every \(\lambda\in[m,M]\) satisfies
\[
    |1-h\lambda|\leq1-hm.
\]
Hence
\begin{equation}
\label{eq:matrix-contraction-bound}
    \|I_3-hQ\|_2\leq r_h,
    \qquad
    r_h:=1-hm\in[0,1).
\end{equation}

For \(\bz\in\overline B_j\), the fundamental theorem of calculus
gives
\[
    \nabla\widetilde q(\bz)
    =
    \nabla\widetilde q(\bx_j)
    +
    Q_{\bz,j}(\bz-\bx_j),
\]
where
\[
    Q_{\bz,j}
    :=
    \int_0^1
    \nabla^2\widetilde q
    \bigl(\bx_j+t(\bz-\bx_j)\bigr)\,dt.
\]
The ball is convex, so \eqref{eq:cor-hessian-interval} implies
\(mI_3\preceq Q_{\bz,j}\preceq MI_3\). Therefore
\[
    T_h(\bz)-\bx_j
    =
    (I_3-hQ_{\bz,j})(\bz-\bx_j)
    -
    h\nabla\widetilde q(\bx_j),
\]
and
\[
    |T_h(\bz)-\bx_j|
    \leq
    r_h|\bz-\bx_j|
    +
    h|\nabla\widetilde q(\bx_j)|.
\]
Since \(q\geq0\) and \(q(\bx_j)=0\), one has
\(\nabla q(\bx_j)=0\). By \cref{lem:perturbation},
\[
    |\nabla\widetilde q(\bx_j)|
    \leq
    \frac{2}{\sqrt3}\kappa\epssub.
\]
The first branch of \eqref{eq:noise}, together with
\(\gamma=\kappa\rho\), is exactly
\begin{equation}
\label{eq:gradient-smaller-than-restoring-force}
    |\nabla\widetilde q(\bx_j)|
    <
    C_0^-\kappa^2\rho
    =
    m\rho.
\end{equation}
Consequently, for \(\bz\in\overline B_j\),
\[
\begin{aligned}
    |T_h(\bz)-\bx_j|
    &\leq
    (1-hm)\rho
    +
    h|\nabla\widetilde q(\bx_j)|\\
    &<
    (1-hm)\rho+hm\rho
    =
    \rho.
\end{aligned}
\]
This proves invariance.

For \(\bz,\bw\in\overline B_j\), define instead
\[
    Q_{\bz,\bw}
    :=
    \int_0^1
    \nabla^2\widetilde q
    \bigl(\bw+t(\bz-\bw)\bigr)\,dt.
\]
Convexity again gives
\(mI_3\preceq Q_{\bz,\bw}\preceq MI_3\), and hence
\[
\begin{aligned}
    |T_h(\bz)-T_h(\bw)|
    &=
    |(I_3-hQ_{\bz,\bw})(\bz-\bw)|\\
    &\leq
    r_h|\bz-\bw|.
\end{aligned}
\]
Thus \(T_h\) is a strict contraction of the complete metric space
\(\overline B_j\) into itself. The Banach fixed-point theorem gives a
unique fixed point and the estimate \eqref{eq:iterate-linear-rate}.
Because \(h>0\), a fixed point of \(T_h\) is a critical point of
\(\widetilde q\). By \cref{lem:local_minima_matched}, it is precisely
\(\widetilde\bx_j\).

Finally, Taylor's formula on the segment joining
\(\widetilde\bx_j\) and \(\bz^{(n)}\), together with
\(\nabla\widetilde q(\widetilde\bx_j)=0\) and the Hessian upper bound,
gives
\[
    0
    \leq
    \widetilde q(\bz^{(n)})-\widetilde q(\widetilde\bx_j)
    \leq
    \frac M2|\bz^{(n)}-\widetilde\bx_j|^2.
\]
Combining this with \eqref{eq:iterate-linear-rate} proves
\eqref{eq:objective-linear-rate}. The final assertions follow by
minimizing
\(\max\{|1-hm|,|1-hM|\}\), whose two endpoint terms agree at
\(h=2/(m+M)\).
\end{proof}

\subsection{Proof of \cref{thm:certified-music-landscape}}\label{sec:thm-proof}

\begin{proof}
Collecting the above results, we have the following.

By the first branch of  the noise condition in \eqref{eq:noise},
$$
C^-_0=	\frac23
			-
			C_2
			-
			\left(
			\frac{2}{\sqrt5}+\frac23\right) 	\eps_{\rm sub} > \frac{
            \frac23-C_2
       }{1+\frac{\sqrt3}{2}\gamma \left(\frac{2}{\sqrt5}+\frac23\right)}>0
       $$
and hence \cref{lem:hessian} implies \cref{thm:certified-music-landscape} (i). 

\cref{lem:outside_lower_matched} implies \cref{thm:certified-music-landscape} (ii);

Solving for $\epssub$ from the inequality implicit in \eqref{eq:9.1} and the definition of $C_0^-$ gives exactly the first branch of  the noise condition in \eqref{eq:noise} and hence
\cref{lem:local_minima_matched} implies \cref{thm:certified-music-landscape} (iii). 

\cref{lem:threshold} and \cref{lem:music-global-lipschitz} imply \cref{thm:certified-music-landscape} (iv).

\cref{thm:convergence} implies \cref{thm:certified-music-landscape} (v).

The proof of \cref{thm:certified-music-landscape} is complete. 

\end{proof}

\subsection{Proof of \cref{thm:localization-modulus} }\label{sec:localization-modulus}
\begin{proof}
We first prove the upper bound.  If
\(X\in\mathfrak X_{s,\kappa}(\gamma,\overline\mu)\), then
\(\nu_X\leq\overline\nu\).  Every term defining \(C_1\) and
\(C_2\) is nondecreasing in \(\mu_X\) and \(\nu_X\), so
\[
    C_2\leq\overline C_2.
\]
Moreover, for
\[
    F(\mu,\nu):=\frac{b_\gamma-\mu-\nu}{1-\mu}
\]
one has
\[
    \partial_\mu F
    =\frac{b_\gamma-1-\nu}{(1-\mu)^2}\leq0,
    \qquad
    \partial_\nu F=-\frac1{1-\mu}<0.
\]
Thus \(F\) decreases in each of
\(\mu\in[0,1)\) and \(\nu\geq0\), and hence
\begin{equation}
\label{eq:oracle-A-uniformization}
    A_\gamma\geq A_\star.
\end{equation}
The definition of \(\eta_{\rm loc}\) is exactly the solution of
\[
    \eta<\frac{\sqrt3}{2}\gamma
       \left(\frac23-\overline C_2-\left(\frac{2}{\sqrt5}+\frac23\right)\eta\right)
    =\frac{\sqrt3}{2}\gamma c_\eta,
\]
while \(\eta<\eta_{\rm gap}\) is exactly the sharp angular
threshold condition for \(A_\star\).  Thus the proof of
\cref{thm:certified-music-landscape}, used with the conservative constants
\(c_\eta\) and \(A_\star\), yields one unique minimum
\(\widetilde\bx_j\) in every \(B_\rho(\bx_j)\), with
\begin{equation}
\label{eq:oracle-upper-localization}
    |\widetilde\bx_j-\bx_j|
    \leq\frac{2}{\sqrt3c_\eta}\frac{\eta}{\kappa}.
\end{equation}

This also defines an estimator without knowledge of \(X\).  Given
\(\widetilde{\mathcal U}\), all local minima of $ \widetilde q$ have value strictly below
$$  \tau_\eta
    :=\sin^2\left(
       \arcsin\sqrt{A_\star}-\arcsin\eta
    \right),$$  whenever there are exactly \(s\) of them.  Otherwise
return an arbitrary fixed \(s\)-point cloud.  Each relevant minimum
has value at most \(\eta^2<\tau_\eta\), whereas
\eqref{eq:angular-outside-lower} and
\eqref{eq:oracle-A-uniformization} imply
\(\widetilde q \geq\tau_\eta\) outside all
the wells.  Strong convexity gives exactly one critical point inside
each well.  Hence  \eqref{eq:oracle-upper-localization}
proves \eqref{eq:oracle-upper-bound}.

For the lower bound, let \(\Phi_t:\mathbb C^s\to\mathcal H\) be the
synthesis operator associated with \(X_t\).  The frame lower bound
along \eqref{eq:oracle-displacement-path} gives
\[
    \sigma_{\min}(\Phi_t)^2
    =\lambda_{\min}(G_{X_t})
    \geq1-\mu_{X_t}\geq 1-\overline\mu.
\]
Only the first column of \(\Phi_t\) varies, and
\begin{align*}
    \|\Phi_t-\Phi_0\|_2^2
    &\leq
    \|\varphi_{\bx_1+t\bv}-\varphi_{\bx_1}\|_2^2=\frac1{4\pi}\int_{\mathbb S^2}
       |e^{\ii\kappa t\bm\omega\cdot\bv}-1|^2\,d\bm\omega\leq\frac{\kappa^2t^2}{3}.
\end{align*}
For equal-rank, full-column-rank operators \(A,B\), the two
one-sided gap estimates give
\begin{equation}
\label{eq:synthesis-projector-perturbation}
    \|P_{\operatorname{Ran}A}-P_{\operatorname{Ran}B}\|_2
    \leq
    \frac{\|A-B\|_2}
     {\min\{\sigma_{\min}(A),\sigma_{\min}(B)\}}.
\end{equation}
Indeed, for \(y=Bc\in\operatorname{Ran}B\), \(\|y\|=1\),
\[
    \|(I-P_{\operatorname{Ran}A})y\|_2
    \leq\|A-B\|_2\,\|B^\dagger\|_2,
\]
and the reverse gap follows by interchanging \(A\) and \(B\).
For equal-dimensional subspaces the two one-sided gaps equal the
projector distance.  Applying \eqref{eq:synthesis-projector-perturbation}
to \(\Phi_0,\Phi_t\) yields
\begin{equation}
\label{eq:path-projector-distance}
    \|P_{\mathcal U_{X_t}}-P_{\mathcal U_{X_0}}\|_2
    \leq\frac{\kappa t}{\sqrt{3(1-\overline\mu)}}.
\end{equation}

Choose \(t=\sqrt{3(1-\overline\mu)}\eta/\kappa\).  The second condition in
\eqref{eq:oracle-two-sided-range} ensures \(t\leq t_0\).  The single
observation \(\widetilde{\mathcal U}=\mathcal U_{X_t}\) is then
compatible both with the truth \(X_t\), with zero error, and with
the truth \(X_0\), with error at most \(\eta\).  Since
\(t\leq t_0\leq\delta_{X_0}/4\), the identity matching has cost
\(t\), while every nonidentity matching has cost greater than \(t\).
Thus \(d_{\rm match}(X_0,X_t)=t\).  For every estimator
\(\widehat X\), the triangle inequality gives
\[
    \max\left\{
       d_{\rm match}(\widehat X(\widetilde{\mathcal U}),X_0),
       d_{\rm match}(\widehat X(\widetilde{\mathcal U}),X_t)
    \right\}
    \geq\frac t2
    =\frac{\sqrt{3(1-\overline\mu)}}2\frac{\eta}{\kappa}.
\]
Taking the infimum over estimators proves the lower half of
\eqref{eq:oracle-two-sided-bound}; the upper half was already proved.
\end{proof}

	\section{Threshold-initialized MUSIC algorithm}\label{sec:algorithm}

We now turn MUSIC optimization theorems into a completely specified
initialization-and-descent procedure.  The formulation below permits
certified upper and lower bounds in place of the exact cloud-dependent
constants.

Suppose we are given an upper bound $\eta$ of $\epssub$, i.e. $\epssub\le\eta$, and 
suppose that an a priori parameter calculation supplies numbers
\(\Delta_\star\) and \(A_\star\) such that
\begin{equation}
\label{eq:algorithm-certificates}
    C_2
      +\left(\frac{2}{\sqrt5}+\frac23\right)\eta
    \leq\Delta_\star<\frac23,
    \qquad
    0<A_\star\leq A_\gamma. 
\end{equation}
Define
\begin{equation}
\label{eq:algorithm-certified-constants}
    c^-_\star:=\frac23-\Delta_\star,
    \qquad
    c^+_\star:=\frac23+\Delta_\star,
    \qquad
    \eps_\star
    :=\frac{2\eta}{\sqrt3c^-_\star\kappa},
\end{equation}
and assume
\begin{equation}
\label{eq:algorithm-noise-certificates}
    \eta<\frac{\sqrt3}{2}\gamma c^-_\star,
    \qquad
    \eta<
    \sin\left(\frac12\arcsin\sqrt{A_\star}\right).
\end{equation}
The corresponding certified outside threshold and optimal scalar step
are
\begin{equation}
\label{eq:algorithm-threshold-step}
    \tau_\star
    :=
    \sin^2\left(
        \arcsin\sqrt{A_\star}-\arcsin\eta
    \right),
    \qquad
    h_\star
    :=\frac{2}{(c^-_\star+c^+_\star)\kappa^2}
    =\frac{3}{2\kappa^2}.
\end{equation}
The associated certified contraction factor is
\begin{equation}
\label{eq:algorithm-contraction-factor}
    r_\star
    :=\frac{c^+_\star-c^-_\star}
            {c^+_\star+c^-_\star}
    =\frac32\Delta_\star<1.
\end{equation}
Notice that \eqref{eq:algorithm-noise-certificates} gives both
\(\eps_\star<\rho\) and \(\eta^2<\tau_\star\).

Let \(G\subset B_1(0)\) be a finite grid satisfying
\begin{equation}
\label{eq:music-grid-condition}
    \mesh_{B_1}(G)
    <
    \min\left\{
       \rho-\eps_\star,
       \frac{\sqrt3}{2\kappa}
       \bigl(\tau_\star-\eta^2\bigr)
    \right\}.
\end{equation}
For the observed projector \(P_{\widetilde{\mathcal U}}\), form
\[
    \widetilde q(\bz)
    =1-\|P_{\widetilde{\mathcal U}}\varphi_\bz\|_2^2,
    \qquad
    G_\star:=\{\bm g\in G:\widetilde q(\bm g)<\tau_\star\}.
\]
For a certified finite output, let \(0<\delta_\star\leq\delta_X\)
be any available separation lower bound, and choose \(N\) and a
clustering radius \(r_{\rm cl}\) such that
\begin{equation}
\label{eq:algorithm-clustering-certificate}
\begin{split}
    8\rho r_\star^N
    &<\delta_\star-2\eps_\star,\\
    4\rho r_\star^N
    &<r_{\rm cl}
      <\delta_\star-2\eps_\star-4\rho r_\star^N.
\end{split}
\end{equation}
The derivative needed by the descent step is available directly from
the projector: for every direction \(\bv\),
\begin{equation}
\label{eq:music-directional-gradient}
    \partial_\bv\widetilde q(\bz)
    =-2\operatorname{Re}
      \langle
        P_{\widetilde{\mathcal U}}
        \partial_\bv\varphi_\bz,
        P_{\widetilde{\mathcal U}}\varphi_\bz
      \rangle,
    \qquad
    \partial_\bv\varphi_\bz(\bm\omega)
    =\ii\kappa(\bm\omega\cdot\bv)\varphi_\bz(\bm\omega).
\end{equation}

\begin{algorithm}[H]
\caption{Certified coarse-grid spherical MUSIC}
\label{alg:coarse-grid-music}
\begin{algorithmic}[1]
\Require \(P_{\widetilde{\mathcal U}},\kappa,G,\eta,
          \gamma,\Delta_\star,A_\star,\delta_\star,N,r_{\rm cl}\),
          satisfying \eqref{eq:algorithm-certificates}, \eqref{eq:algorithm-noise-certificates}, \eqref{eq:algorithm-contraction-factor}, \eqref{eq:algorithm-clustering-certificate}
\State Compute \(c^-_\star,c^+_\star,\eps_\star,
                  \tau_\star,h_\star,r_\star\) from
       \eqref{eq:algorithm-certified-constants} and
       \eqref{eq:algorithm-threshold-step}--
       \eqref{eq:algorithm-contraction-factor}
\State \(G_\star\gets
       \{\bm g\in G:\widetilde q(\bm g)<\tau_\star\}\)
\State \(Y_N\gets\varnothing\)
\ForAll{\(\bm g\in G_\star\)}
    \State \(\bz^{(0)}\gets\bm g\)
    \For{\(n=0,\ldots,N-1\)}
        \State \(\bz^{(n+1)}\gets
          \bz^{(n)}-h_\star\nabla\widetilde q(\bz^{(n)})\)
    \EndFor
    \State \(Y_N\gets Y_N\cup\{\bz^{(N)}\}\)
\EndFor
\State Join two endpoints when their distance is less than
       \(r_{\rm cl}\)
\State \Return  one representative from each connected component. \end{algorithmic}
\end{algorithm}
It can be shown that under a stronger condition $\delta_X> 4\rho$, there are exactly $s$ geometric clusters and gradient descent initialized from one element from each cluster recovers the relevant minima. 

\begin{theorem}[Correctness of the certified algorithm]
\label{thm:coarse-grid-music-correctness}
Assume the geometric hypotheses of
\cref{thm:certified-music-landscape}. Let
\[
    \eta\geq\eps_{\rm sub},
    \qquad
    0<\delta_\star\leq\delta_X.
\]
Assume the parameter and noise certificates
\eqref{eq:algorithm-certificates}--%
\eqref{eq:algorithm-noise-certificates},
the grid condition \eqref{eq:music-grid-condition}, and the finite
clustering certificate
\eqref{eq:algorithm-clustering-certificate}.
Then
\begin{enumerate}[label=\textup{(\roman*)}]
\item every accepted point lies in a certified well:
\[
    G_\star\subset
    \bigcup_{j=1}^sB_\rho(\bx_j);
\]
\item every well contains at least one accepted point;
\item for every \(\bm g\in G_\star\), the iterates in
\cref{alg:coarse-grid-music} remain in one well and converge to its
unique minimizer;
\item after exact duplicate removal, the set of limiting outputs is
precisely \(\{\widetilde\bx_1,\ldots,\widetilde\bx_s\}\), and
\begin{equation}
\label{eq:algorithm-matching-error}
    d_{\rm match}
      \bigl(\{\widetilde\bx_1,\ldots,\widetilde\bx_s\},X\bigr)
    \leq\eps_\star
    =\frac{2\eta}{\sqrt3c^-_\star\kappa}.
\end{equation}

\end{enumerate}
An accepted initialization in the \(j\)-th well obeys the
finite-iteration estimate
\begin{equation}
\label{eq:algorithm-finite-iteration-error}
    |\bz^{(N)}-\widetilde\bx_j|
    \leq2\rho\,r_\star^N.
\end{equation}
Moreover, the finite procedure in \cref{alg:coarse-grid-music}
returns exactly \(s\) representatives \(\widehat X_N\), and
\begin{equation}
\label{eq:algorithm-finite-matching-error}
    d_{\rm match}(\widehat X_N,X)
    \leq\eps_\star+2\rho r_\star^N.
\end{equation}
\end{theorem}

\begin{remark}[Finite stopping and duplicate removal]
Condition \eqref{eq:algorithm-clustering-certificate} is a sufficient,
fully checkable finite-stopping rule.  If no quantitative lower bound
\(\delta_\star\) is available, one may omit the clustering step,
retain the endpoint set \(Y_N\), and use
\eqref{eq:algorithm-finite-iteration-error} to monitor convergence.
The set of distinct trajectory limits still consists exactly of the
\(s\) relevant minima.
\end{remark}
\begin{proof}
The Hessian certificates give, throughout each certified well,
\[
    c^-_\star\kappa^2I_3
    \preceq\nabla^2\widetilde q
    \preceq c^+_\star\kappa^2I_3.
\]
The perturbation estimate at a true atom and the first inequality in
\eqref{eq:algorithm-noise-certificates} give
\[
    |\nabla\widetilde q(\bx_j)|
    \leq\frac{2\kappa\eta}{\sqrt3}
    <c^-_\star\kappa^2\rho.
\]
Thus the invariant-well proof of \cref{thm:convergence}, with the
certified interval in place of the exact interval, applies to
\(h_\star\).  It gives invariance, contraction, and
\eqref{eq:algorithm-contraction-factor}.  It also gives
\eqref{eq:algorithm-finite-iteration-error}, because both the
initial point and the minimizer belong to the same radius-\(\rho\)
ball.

By \eqref{eq:angular-outside-lower},
\(A_\gamma\geq A_\star\), and \(\epssub\leq\eta\), every point
outside the union of the wells satisfies
\[
    \widetilde q(\bz)\geq\tau_\star.
\]
This proves \textup{(i)}.  On the other hand, each relevant minimizer
has value at most \(\epssub^2\leq\eta^2\).  The global Lipschitz
bound and \eqref{eq:music-grid-condition} therefore allow
\cref{lem:threshold} to be applied with
\[
    (\eps,\tau_0,\tau_1,L)
    =\left(
       \eps_\star,\eta^2,\tau_\star,
       \frac{2\kappa}{\sqrt3}
      \right).
\]
It follows that every well contains an accepted grid point, proving
\textup{(ii)}.  The contraction result proves \textup{(iii)}, and
then \textup{(iv)} follows after duplicate limits are removed.  The
localization part of \cref{lem:local_minima_matched} gives
\eqref{eq:algorithm-matching-error}.

It remains only to justify the finite clustering step.  Two endpoints
converging to the same minimizer are, by
\eqref{eq:algorithm-finite-iteration-error}, at distance at most
\(4\rho r_\star^N\).  If two endpoints converge to minima in
different wells, their distance is at least
\[
    \delta_X-2\eps_\star-4\rho r_\star^N
    \geq
    \delta_\star-2\eps_\star-4\rho r_\star^N.
\]
The two strict inequalities in
\eqref{eq:algorithm-clustering-certificate} therefore make the graph
components coincide exactly with the \(s\) well groups.  Every chosen
representative from the \(j\)-th group lies within
\(2\rho r_\star^N\) of \(\widetilde\bx_j\), which proves
\eqref{eq:algorithm-finite-matching-error}.
\end{proof}

\section{Numerical experiment}\label{sec:num}

The target 125-point clouds are generated from $5\times 5\times 5$-scaffold of pitch \(0.20\), with isotropic radial jitter between \(0.05\) and \(0.1\), centering of the jitters, a random Haar rotation, and a translation of norm at most \(0.02\), resulting in minimum separations
\beq
\label{eq:delta-range}
 0.03251235\leq\delta_X\leq0.07315914.
 \eeq
The measured frequency range across all clouds and sweep values is
\beq
\label{eq:k-range}
8093.39\leq\kappa\leq346022.37.
\eeq

For \(\gamma=7/12\), the certificate transitions occur at approximately
\[
\varepsilon_{\rm sub}<0.0477
\quad\text{when }\mu_X=0.1,
\] and
\[
\mu_X<0.0715
\quad\text{when }\varepsilon_{\rm sub}=0.1.
\]
The gray zones of the plots are uncertified, but they are still reconstructed by using the empirical threshold
\[
\tau=\varepsilon_{\rm sub}^{\,2}+10^{-3}.
\]
The experiment does not use an adversarial complement aligned with the spatial derivatives of the target Fourier atoms. Instead, the perturbation is constructed using 250 auxiliary Fourier atoms whose spatial centers lie in the ball of radius \(0.9\), with each auxiliary center at least \(0.04\) from every target and from every other auxiliary center. It forms \(s=125\) random complex mixtures of these atoms, projects those mixtures onto the orthogonal complement of the noiseless signal subspace, and re-orthonormalizes them to obtain an orthonormal complementary basis \(W\), satisfying
\(U^*U=W^*W=I_s,\) \( U^*W=0.\)
The observed subspace basis is then obtained by the equal-angle rotation
\[
\widetilde U
=
\sqrt{1-\varepsilon_{\rm sub}^2}\,U
+
\varepsilon_{\rm sub}W.
\]
We compare typical localization errors for a structured random-complement perturbation with the deterministic finite-output upper certificate \eqref{eq:algorithm-finite-matching-error} with $N=12$. 
\begin{figure}[ht]
\centering
\begin{subfigure}{0.45\textwidth}
\includegraphics[width=1\textwidth]{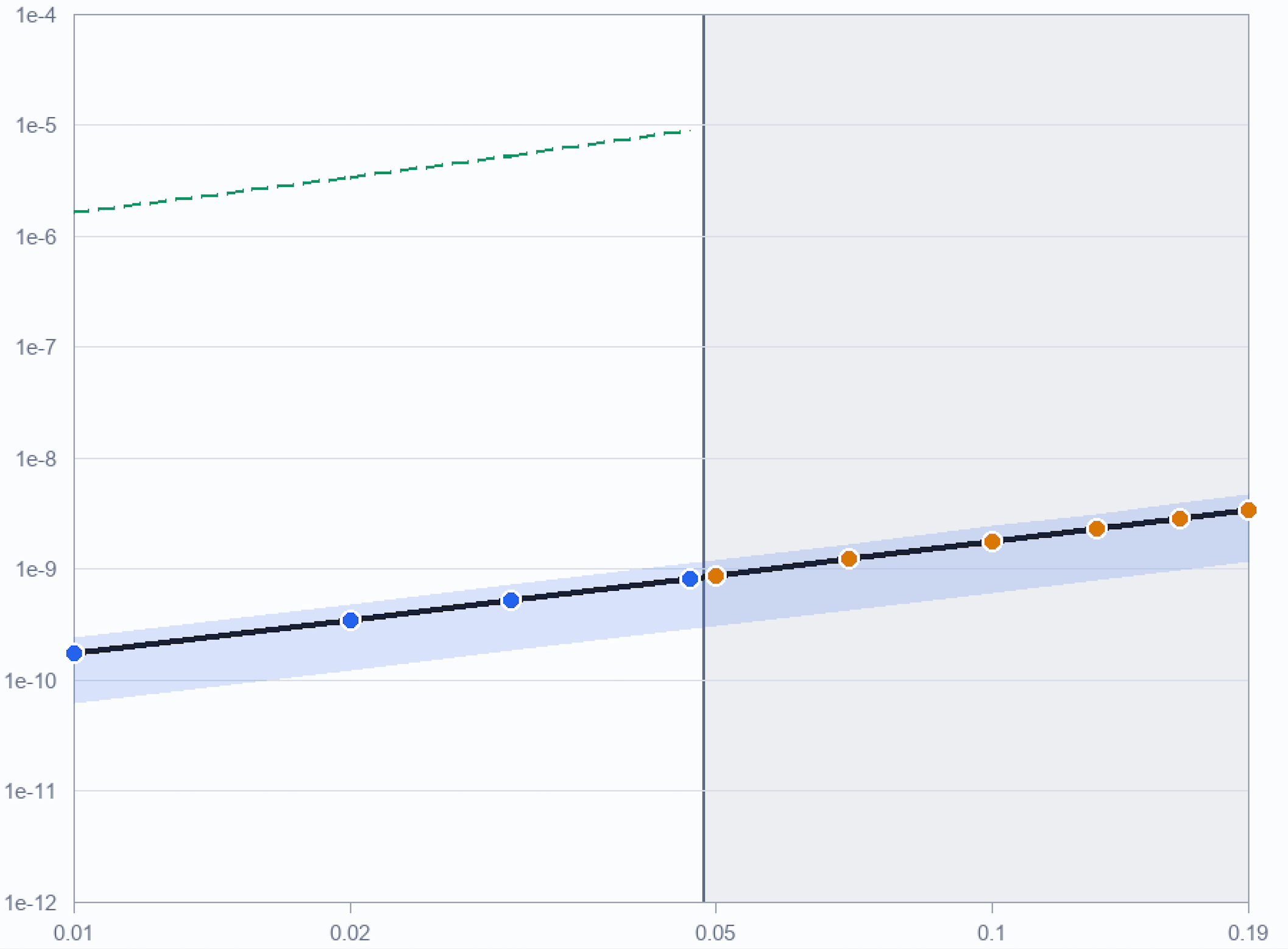}
\caption{$\mu_X=0.1$}
\end{subfigure}\qquad
\begin{subfigure}{0.45\textwidth}
\includegraphics[width=1\textwidth]{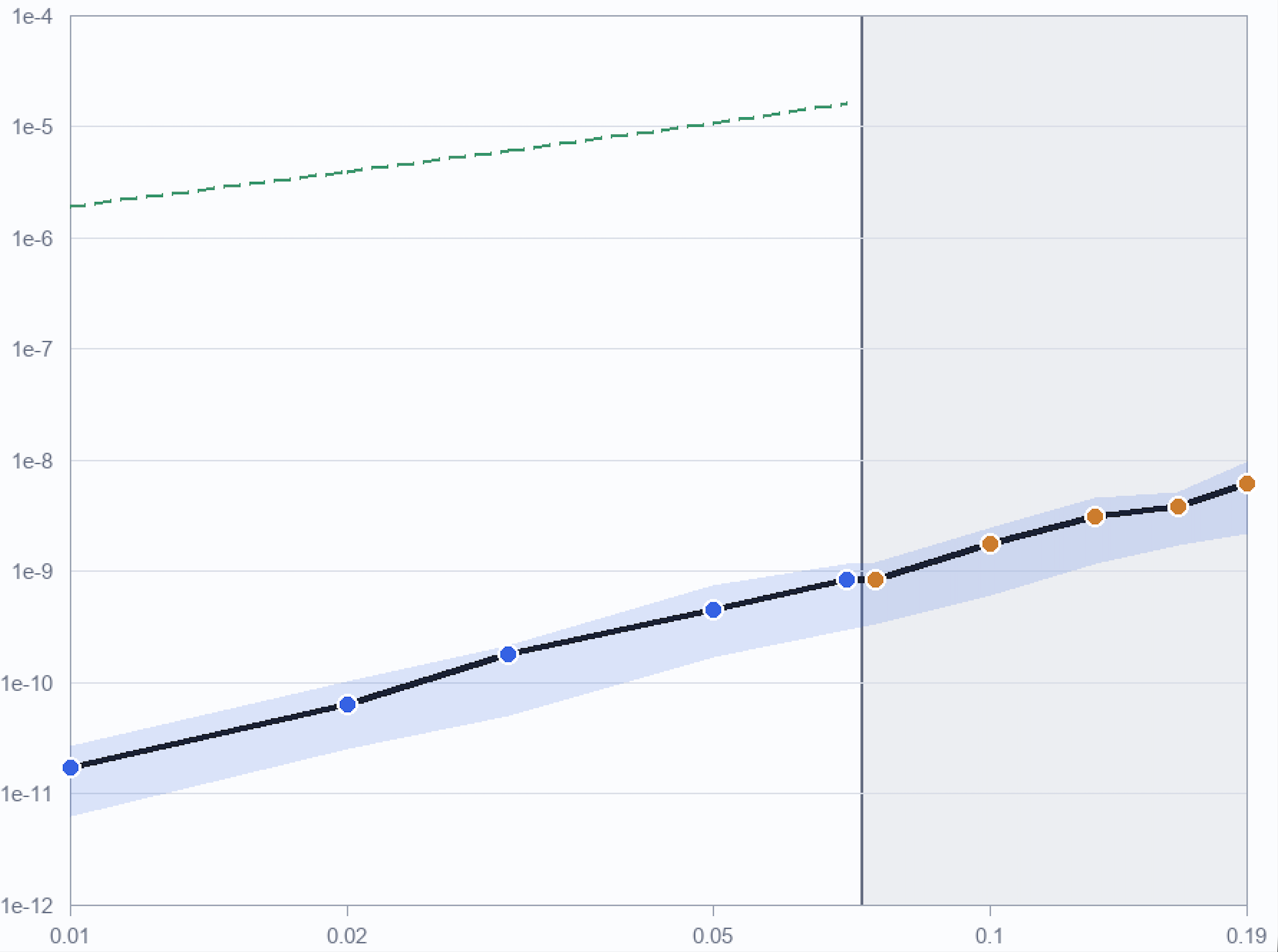}
\caption{$\epssub=0.1$}
\end{subfigure}
\caption{log-log plot of localization error for $s=125$ versus (left) $\epssub\in [0.01,0.19]$ with fixed $\mu_X=0.1$ \& (right) $\mu_X\in [0.01, 0.19]$ with fixed $\epssub=0.1, $ both with $\gamma=7/12,$ 12 gradient descent steps.  The wavenumber is $\kappa=112.5/(\mu_X \delta_X)$. 10  trials  are performed for each case with recovery rate $=1$ in all cases. The (blue and orange) solid points  and the solid lines connecting them are the median localization error and the dotted lines are the certified error bound. The gray zones are   certificate-infeasible. The blue bands indicate the full range of 10 trials.}
\label{fig:numerical-error}
\end{figure}

A log–log least-squares fit to the median final errors  in \cref{fig:numerical-error} gives
\beq
\operatorname{median\,\,error}
\asymp
\varepsilon_{\rm sub}^{\,1.00547},&
\qquad R^2=0.999986,&\qquad \hbox{when}\,\,\mu_X=0.1
\eeq and
\beq
\operatorname{median\,\,error}
\asymp
\mu_X^{\,1.98223},&
\qquad R^2=0.997971, &\qquad \hbox{when}\,\,\epssub=0.1.
\eeq Thus the reconstructed locations show essentially linear dependence on \(\varepsilon_{\rm sub}\) and quadratic dependence on \(\mu_X\), across both certified and uncertified regimes.

An unsatisfactory feature of the present work is that under the packing condition $\kappa\delta_X\gtrsim s^{2/3}$ the number of points in a certified thresholding grid  for a large $s$ is enormous, rendering Algorithm 1 computationally  impractical. One possible solution is via adaptive gridding instead of a uniform grid. This is beyond the scope of the paper.

\subsection{Heuristic explanation for the observed quadratic scaling}
Near a true target \(\bx_j\), write the perturbed MUSIC objective as a quadratic bowl plus a noise-induced tilt. For a small displacement \(d\),
\[
\nabla\widetilde q(\bx_j+\bd)
\approx
\bg_j+H(\bx_j)d,
\] where
\[
\bg_j:=\nabla\widetilde q(\bx_j)
\]is the perturbation-induced gradient at the true target.
The manuscript proves that the curvature of the MUSIC well is of order
\[
H(\bx_j)\asymp \kappa^2I.
\] This is the Hessian bound. The displaced local minimum \(\widetilde \bx_j=\bx_j+\bd_j\) therefore approximately satisfies
\[
0=\bg_j+H(\bx_j)\bd_j,
\] so that
\[
\bd_j\approx-H(\bx_j)^{-1}\bg_j,
\qquad
|\bd_j|\asymp\frac{|\bg_j|}{\kappa^2}.
\]Thus the power of \(\mu_X\) is determined by how the realized perturbation gradient \(\bg_j\) depends on \(\kappa\).

For a true atom \(\varphi_{\bx_j}=Ua_j\), orthogonality gives \(W^*\varphi_{\bx_j}=0\). Differentiating the perturbed objective gives, in any direction \(v\),
\[
\partial_v\widetilde q(\bx_j)
=
-2\varepsilon_{\rm sub}
 \sqrt{1-\varepsilon_{\rm sub}^2}\,
 \operatorname{Re}
 \left\langle
 W^*\partial_v\varphi_{\bx_j},a_j
 \right\rangle.
\]The size of the gradient is therefore governed by derivative correlations between the target atoms and the auxiliary atoms.
For two separated locations with distance \(r>0\), the spherical Fourier kernel is
\[
K_\kappa(r)=\frac{\sin(\kappa r)}{\kappa r},
\] and
\[
\left|K_\kappa'(r)\right|
=
\left|
\frac{\kappa r\cos(\kappa r)-\sin(\kappa r)}
     {\kappa r^2}
\right|
\leq
\frac1r+\frac1{\kappa r^2}.
\]Because the auxiliary centers stay a fixed positive distance from the targets, these derivative correlations are approximately \(O(1)\) as \(\kappa\) increases—not \(O(\kappa)\). The high-frequency oscillations decorrelate the separated Fourier atoms.
Thus, for this particular random complement,
\[
|\bg_j|
\approx O(\varepsilon_{\rm sub}),
\]rather than the worst-case \(O(\kappa\varepsilon_{\rm sub})\). Since the Hessian remains \(O(\kappa^2)\),
\[
\eps_{\rm loc}
\asymp
\frac{|\bg_j|}{\kappa^2}
\approx
O\!\left(\frac{\varepsilon_{\rm sub}}{\kappa^2}\right)
\] 
and therefore a quadratic dependence on $\mu_X$. 

\section{Conclusion and outlook}
\label{sec:conclusion}

We have established a quantitative theory for recovering a finite
three-dimensional point cloud from a perturbed spherical Fourier
subspace. Under explicit separation, boundary-containment, conditioning,
and noise hypotheses, the MUSIC objective has one strongly convex well
near each scatterer and a uniform value gap outside the certified wells.
These properties connect a geometric stability estimate to a complete
localization procedure: thresholding supplies an initialization in every
well, a fixed-step gradient iteration remains in that well and converges
linearly, and certified clustering removes duplicate outputs. The
resulting finite-iteration error separates the subspace uncertainty from
the optimization error,
\[
    d_{\rm match}(\widehat X_N,X)
    \leq
    \frac{2\eta}{\sqrt{3}\,c^-_\star\kappa}
    +2\rho r_\star^N,
\]
with the quantities and hypotheses of
\cref{thm:coarse-grid-music-correctness}. On uniformly admissible classes
containing a one-point displacement path,
\cref{thm:localization-modulus} also gives the matching lower bound
\[
    \mathfrak R_{\mathfrak X}(\eta)
    \asymp \frac{\eta}{\kappa}
    \qquad (\eta\downarrow0).
\]
Thus the dependence on subspace error and wavenumber is optimal in the
deterministic oracle model.

The scope of this conclusion is tied to the observation model. The
localization stage is incidence blind because the incident fields and
the Foldy--Lax interactions enter through the effective coefficient
matrix, while the scatterer locations determine the spherical Fourier
subspace. Multiple scattering is retained in this factorization.
Recovering the full signal subspace nevertheless requires sufficiently
diverse effective illuminations, and its stable estimation depends on
their conditioning and on the measurement process. The present
localization theorem therefore supplies the geometric part of a
measurement-to-location analysis. A natural next step is to derive
quantitative subspace-error certificates from finite, noisy scattering
data, accounting for receiver sampling, aperture, calibration, and
illumination diversity. In particular, assessing the benefit of larger
\(\kappa\) at a fixed measurement budget requires controlling how
\(\varepsilon_{\rm sub}\) itself depends on frequency and acquisition.

An unresolved analytical question is the optimal separation scale for
uniform spherical frame bounds. The sufficient condition
\(\kappa\delta_X\gtrsim s^{2/3}\) follows from an absolute coherence row
sum, and \cref{app:row-sum-example} shows that its exponent cannot be
improved within that argument. The lower- and upper-frame obstructions
in \cref{app:lower-example,app:upper-example} occur at different scales:
they rule out uniform lower bounds below the \(s^{1/6}\) scale and uniform
upper bounds below the \(s^{1/3}\) scale. The optimal exponent for
two-sided arbitrary-cloud conditioning consequently remains unresolved
between \(1/3\) and \(2/3\). Improving the sufficient scale would require
spectral estimates that retain information discarded by absolute row
sums. Transferring such an improvement to the full localization theorem
would also require compatible control of the derivative correlations
and the objective away from the targets.

The principal computational limitation is the cost of certified global
initialization. For fixed dimensionless certificate margins, a uniform
mesh of spacing \(O(\kappa^{-1})\) in a fixed three-dimensional search
region requires \(O(\kappa^3)\) points. This cost can be prohibitive in
the separation regime considered here, even though refinement within
each well is contractive. Adaptive subdivision is therefore a natural
direction: bounds on the MUSIC objective over a cell could exclude
regions that cannot meet the acceptance threshold and concentrate
refinement near candidate wells. A rigorous replacement for the uniform
grid would need to preserve coverage of every certified well and provide
an explicit complexity bound. Establishing such a guarantee remains
open.

The numerical experiment also motivates a distinction between
adversarial perturbations and structured perturbation models. The
observed nearly quadratic dependence on \(\mu_X\), at fixed subspace
error, is consistent with the heuristic
\(O(\varepsilon_{\rm sub}/\kappa^2)\) displacement for the separated
auxiliary-atom construction. This behavior is compatible with the
worst-case \(O(\varepsilon_{\rm sub}/\kappa)\) modulus because the
experiment samples a restricted family of perturbation directions.
A rigorous explanation would require quantitative bounds on the
derivative correlations after projection and orthonormalization, together
with control of the local expansion and finite-iteration error. Such an
analysis could clarify which acquisition-induced perturbations attain
the adversarial rate and which permit better accuracy.

A natural extension concerns measurements with the angular variable restricted to a nonempty open subset of $\bmS^2$ such as a hemisphere.  Under suitable conditioning, curvature, and value-gap estimates, the subspace perturbation argument would yield analogous guarantees, with constants reflecting the aperture geometry. Establishing explicit estimates of this kind, and quantifying their deterioration as the aperture narrows, are natural directions for future work. 

Finally, extension to extended scatterers or continuously distributed
contrasts requires an analogue of the finite-cloud separation and
conditioning hypotheses. A continuum limit cannot be obtained by simply
increasing the number of points at fixed frequency, since the present
uniform certificates need not persist as the cloud becomes dense.
Identifying an appropriate geometric scale or effective complexity, and
determining whether it supports a comparable subspace-to-geometry
stability principle, remain open problems.

  	\appendix
	\section{Positivity of the Gram matrix}
	\label{app:gram}
\begin{lemma}
\label{lem:gram-positive}
Let $\kappa>0$, and let
\[
    X=\{\bm x_1,\ldots,\bm x_s\}\subset\mathbb R^3
\]
be a finite set of pairwise distinct points.  Define
\[
    G_X
    :=
    \left[
        K_\kappa(\bm x_j-\bm x_k)
    \right]_{j,k=1}^s,
    \qquad
    K_\kappa(\bm h)
    =
    \frac{1}{4\pi}
    \int_{\mathbb S^2}
        e^{i\kappa\bm\omega\cdot\bm h}
    \,d\bm\omega.
\]
Then $G_X$ is Hermitian strictly positive definite and hence
invertible.
\end{lemma}

\begin{proof}
For $\bm\alpha=(\alpha_1,\ldots,\alpha_s)^\top\in\mathbb C^s$,
the spherical representation of $K_\kappa$ gives
\[
    \bm\alpha^*G_X\bm\alpha
    =
    \frac{1}{4\pi}
    \int_{\mathbb S^2}
    \left|
        \sum_{j=1}^s
        \alpha_j e^{-i\kappa\bm\omega\cdot\bm x_j}
    \right|^2
    \,d\bm\omega.
\]
In particular, $G_X$ is positive semidefinite.

Suppose that $\bm\alpha^*G_X\bm\alpha=0$, and define
\[
    F_{\bm\alpha}(\bm z)
    :=
    \sum_{j=1}^s
        \alpha_j e^{-i\kappa\bm z\cdot\bm x_j},
    \qquad
    \bm z\in\mathbb C^3.
\]
This finite exponential sum is entire on $\mathbb C^3$.  Since its
restriction to $\mathbb S^2$ is continuous, the preceding integral
identity implies
\[
    F_{\bm\alpha}(\bm\omega)=0
    \qquad
    \text{for every }\bm\omega\in\mathbb S^2.
\]

Choose $\bm b\in\mathbb S^2$ such that the numbers
\[
    \beta_j:=\bm b\cdot\bm x_j,
    \qquad j=1,\ldots,s,
\]
are pairwise distinct.  Such a direction exists because the
exceptional set is contained in the finite union of great circles
\[
    \bigcup_{j\ne k}
    \left\{
        \bm b\in\mathbb S^2:
        \bm b\cdot(\bm x_j-\bm x_k)=0
    \right\}.
\]
After relabeling, suppose that
\[
    \beta_1>\beta_2>\cdots>\beta_s.
\]
Choose $\bm a\in\mathbb S^2$ with $\bm a\perp\bm b$, and set
\[
    H(\zeta)
    :=
    F_{\bm\alpha}
    \bigl(
        \bm a\cos\zeta+\bm b\sin\zeta
    \bigr),
    \qquad
    \zeta\in\mathbb C.
\]
The function $H$ is entire.  For every $\theta\in\mathbb R$,
\[
    \bm a\cos\theta+\bm b\sin\theta\in\mathbb S^2,
\]
and therefore $H(\theta)=0$.  The identity theorem implies that
$H\equiv0$ on $\mathbb C$.

Evaluating at $\zeta=it$, where $t>0$, yields
\[
    0
    =
    \sum_{j=1}^s
    \alpha_j
    \exp\left(
        -i\kappa(\bm a\cdot\bm x_j)\cosh t
        +
        \kappa\beta_j\sinh t
    \right).
\]
After division by
\[
    \exp\left(
        -i\kappa(\bm a\cdot\bm x_1)\cosh t
        +
        \kappa\beta_1\sinh t
    \right),
\]
we obtain
\[
    0
    =
    \alpha_1
    +
    \sum_{j=2}^s
    \alpha_j
    e^{-\kappa(\beta_1-\beta_j)\sinh t}
    e^{-i\kappa
       \bm a\cdot(\bm x_j-\bm x_1)\cosh t}.
\]
The last exponential factor has modulus one, whereas
\[
    e^{-\kappa(\beta_1-\beta_j)\sinh t}\longrightarrow0
    \qquad
    (t\to\infty)
\]
for every $j\ge2$.  Hence $\alpha_1=0$.  Repeating the same argument
for the remaining indices gives
\[
    \alpha_1=\cdots=\alpha_s=0.
\]
Thus the quadratic form vanishes only at the origin, and therefore
$G_X\succ0$.
\end{proof}

\section{ A regular-grid lower-frame bound obstruction}
\label{app:lower-example}

The following  lower-frame obstruction shows 
that every polynomial threshold \(Cs^p\) with \(p<1/6\) is
insufficient, even for regular three-dimensional grids. 

\begin{lemma}[A spherical-Bessel bound]
\label{lem:spherical-bessel-elementary-bound}
For every integer \(\ell\ge0\) and every real \(z\ge0\),
\[
    |j_\ell(z)|
    \le
    \frac{z^\ell}{(2\ell+1)!!}.
\]
If \(\ell\ge1\), then
\[
    |j_\ell(z)|
    \le
    \frac{z^\ell}{(2\ell+1)!!}
    \le
    \left(\frac{ez}{2\ell}\right)^\ell.
\]
\end{lemma}

\begin{proof}
Poisson's integral representation for \(J_{\ell+1/2}\), together
with
\[
    j_\ell(z)
    =
    \sqrt{\frac{\pi}{2z}}J_{\ell+1/2}(z),
\]
gives
\[
    j_\ell(z)
    =
    \frac{z^\ell}{2^{\ell+1}\ell!}
    \int_{-1}^{1}
    e^{izt}(1-t^2)^\ell\,dt.
\]
Since \(z\) is real,
\[
    |e^{izt}|=1,
\]
and therefore
\[
    |j_\ell(z)|
    \le
    \frac{z^\ell}{2^{\ell+1}\ell!}
    \int_{-1}^{1}(1-t^2)^\ell\,dt.
\]
The beta integral satisfies
\[
    \int_{-1}^{1}(1-t^2)^\ell\,dt
    =
    B\!\left(\frac12,\ell+1\right)=
    \frac{\Gamma(\frac12)\Gamma(\ell+1)}
         {\Gamma(\ell+\frac32)}=
    \frac{2^{\ell+1}\ell!}{(2\ell+1)!!}.
\]
Consequently,
\[
    |j_\ell(z)|
    \le
    \frac{z^\ell}{(2\ell+1)!!}.
\]

For \(\ell\ge1\),
\[
    (2\ell+1)!!
    =
    \prod_{k=1}^{\ell}(2k+1)
    \ge
    \prod_{k=1}^{\ell}2k
    =
    2^\ell\ell!.
\]
Using
\[
    \ell!\ge\left(\frac{\ell}{e}\right)^\ell,
\]
we obtain
\[
    \frac{z^\ell}{(2\ell+1)!!}
    \le
    \frac{z^\ell}{2^\ell\ell!}\le
    \frac{z^\ell}
         {2^\ell(\ell/e)^\ell}=
    \left(\frac{ez}{2\ell}\right)^\ell.
\]
\end{proof}

\begin{proposition}
\label{prop:positive-power-lower-frame-obstruction}
Let \(C>0\) and \(0<p<1/6\). There exist point clouds
\[
    X_n\subset\overline{B_{1/2}(0)},
    \qquad
    s_n:=|X_n|\longrightarrow\infty,
\]
and wavenumbers \(\kappa_n\) such that
\[
    \kappa_n\delta_{X_n}=Cs_n^p,
\]
but
\[
    \lambda_{\min}(G_{X_n})\longrightarrow0.
\]
\end{proposition}
\begin{remark} The proof can be slightly modified to show the following.

For every prescribed positive sequence \(a_n\to0\), one can choose the regular cubic grids \(X_n\) and wavenumbers \(\kappa_n\) so that
\[
\kappa_n\delta_{X_n}=a_ns_n^{1/6}
\]while
\[
\lambda_{\min}(G_{X_n})\to0.
\]

\end{remark}
\begin{proof}
For \(n\ge2\), set
\[
    s_n:=n^3,
    \qquad
    \delta_n:=\frac{1}{\sqrt3\,(n-1)},
\]
and define
\[
    X_n
    :=
    \left\{
    \delta_n
    \begin{pmatrix}
        a-\frac{n-1}{2}\\
        b-\frac{n-1}{2}\\
        c-\frac{n-1}{2}
    \end{pmatrix}
    :
    0\le a,b,c<n
    \right\}.
\]
Then
\[
    |X_n|=s_n,
    \qquad
    \delta_{X_n}=\delta_n.
\]
Moreover, for every \(\bm x\in X_n\),
\[
    |\bm x|
    \le
    \delta_n\frac{\sqrt3(n-1)}2
    =
    \frac12,
\]
and hence \(X_n\subset B_{1/2}(0)\).

Choose
\[
    \kappa_n
    :=
    \frac{Cn^{3p}}{\delta_n}.
\]
Then
\[
    \kappa_n\delta_{X_n}
    =
    Cn^{3p}
    =
    Cs_n^p.
\]

Let
\[
    M_n:=\left\lfloor\frac12n^{3/2}\right\rfloor,
    \qquad
    L_n:=M_n-1.
\]
The space of spherical harmonics of degrees \(0,\ldots,L_n\) has
dimension
\[
    (L_n+1)^2=M_n^2\le s_n/4.
\]
Therefore there exists
\(\bm\alpha^{(n)}\in\mathbb C^{s_n}\), with
\(\|\bm\alpha^{(n)}\|_2=1\), such that
\[
    P_{\le L_n}
  \left[  \sum_{\bm x\in X_n}
    \alpha_{\bm x}^{(n)}
    e^{i\kappa_n\bm\bz\cdot\bm x}\right]
    =0.
\]
Write
\[
    F_n(\bm\bz)
    :=
    \sum_{\bm x\in X_n}
    \alpha_{\bm x}^{(n)}
    e^{i\kappa_n\bm\bz\cdot\bm x}.
\]

For a single plane wave, the spherical plane-wave expansion and the
addition theorem give
\[
    \left\|
        (I-P_{\le L})
        e^{i\kappa\bm\bz\cdot\bm x}
    \right\|_2^2
    =
    \sum_{\ell=L+1}^{\infty}
    (2\ell+1)j_\ell(\kappa|\bm x|)^2.
\]

Since \(|\bm x|\le1/2\), put
\[
    \vartheta_n:=\frac{e\kappa_n}{4M_n}.
\]
Because
\beq\label{eq:c.2}
    \kappa_n\asymp n^{1+3p},
    \qquad
    M_n\asymp n^{3/2},
\eeq
we have
\beq
\label{eq:c.1}
    \vartheta_n
    =
    O\left(n^{-(1/2-3p)}\right)
    \longrightarrow0.
\eeq
By \cref{lem:spherical-bessel-elementary-bound}, 
\[
    |j_\ell(\kappa_n|\bm x|)|
    \le
    \vartheta_n^\ell,\quad  \forall \ell\ge M_n,\,\, \bm x\in X_n,
\]
\[
    \left\|
        (I-P_{\le L_n})
        e^{i\kappa_n\bm\bz\cdot\bm x}
    \right\|_2^2
    \le
    (4M_n+6)\vartheta_n^{2M_n}\quad\hbox{for}\quad \vartheta_n\le1/\sqrt 2. 
\]

Since \(P_{\le L_n}F_n=0\), the triangle inequality and
\(\|\bm\alpha^{(n)}\|_1\le\sqrt{s_n}\) give
\[
\begin{aligned}
    (\bm\alpha^{(n)})^*
    G_{X_n}\bm\alpha^{(n)}
    &=
    \|F_n\|_2^2\le
    s_n(4M_n+6)\vartheta_n^{2M_n}\longrightarrow 0
\end{aligned}
\]
 due to \eqref{eq:c.2} and \eqref{eq:c.1}. 
Hence
\[
    \lambda_{\min}(G_{X_n})
    \le
    (\bm\alpha^{(n)})^*
    G_{X_n}\bm\alpha^{(n)}
    \longrightarrow0.
\]
\end{proof}

\section{A thin-box upper-frame bound obstruction}
\label{app:upper-example}
The following example gives a bound on the Rayleigh quotient
of \(G_X\) and proves a necessary worst-case scale \(s^{1/3}\) for a
uniform upper frame bound. It is a highly anisotropic rectangular-grid construction, not a saturated cubic-grid example,
\begin{proposition}
\label{prop:thin-box-upper-frame-obstruction}
There exist absolute constants \(c_0,\eps_0>0\) with the
following property. For every \(0<\eps\le\eps_0\) and
every sufficiently large integer \(n\), there is a point cloud
\(X_n\subset B_1(0)\) with
\[
    s_n:=|X_n|=n^6,
    \qquad
    \delta_{X_n}=n^{-4},
\]
and a wavenumber
\[
    \kappa_n=\eps n^6
\]
such that
\[
    \kappa_n\delta_{X_n}
    =
    \eps n^2
    =
    \eps s_n^{1/3},
\]
whereas the associated spherical Gram matrix satisfies
\[
    \lambda_{\max}(G_{X_n})
    \ge
    \frac{c_0}{\eps}.
\]
Consequently, an arbitrary-cloud spherical frame upper bound cannot hold under a separation threshold
\(o(s^{1/3})\).
\end{proposition}

\begin{proof}
Set
\[
    \delta:=n^{-4}
\]
and define
\[
    X_n
    :=
    \left\{
    \delta
    \begin{pmatrix}
        a-\frac{n^4-1}{2}\\[1mm]
        b-\frac{n-1}{2}\\[1mm]
        c-\frac{n-1}{2}
    \end{pmatrix}
    :
    0\le a<n^4,\quad 0\le b,c<n
    \right\}.
\]
Then
\[
    |X_n|=n^6,
    \qquad
    \delta_{X_n}=\delta.
\]
Moreover,
\[
    |x_1|<\frac12,
    \qquad
    |x_2|,|x_3|<\frac{1}{2n^3},
\]
so \(X_n\subset B_1(0)\).

Set
\[
    \kappa:=\eps n^6
\]
and define a unit coefficient vector by
\[
    \alpha_{\bm x}
    :=
    \frac1{\sqrt{s_n}}e^{-i\kappa x_1},
    \qquad
    \bm x\in X_n.
\]
Then
\[
    \|\bm\alpha\|_2=1.
\]

Parametrize a spherical cap about \(\bm e_1\) by
\[
    \bm\bz(\bm u)
    :=
    \left(\sqrt{1-|\bm u|^2},u_1,u_2\right),
    \qquad
    |\bm u|\le\frac{a_0}{\sqrt{\kappa}},\quad \bu=(u_1,u_2)
\]
where \(a_0>0\) is an absolute constant to be chosen. After the
modulation in \(\alpha_{\bm x}\), the phase of the summand
corresponding to \(\bm x\) is
\[
  \kappa(\bz-\bm e_1)\cdot\bx=  \kappa
    \left[
        \bigl(z_1(\bm u)-1\bigr)x_1
        +
        \bm u\cdot\bm x_\perp
    \right],\quad\bx=(x_1,\bx_\perp),\quad z_1=\sqrt{1-|\bu|^2}
\]
Using
\[
    1-\sqrt{1-|\bm u|^2}\le|\bm u|^2,
    \qquad
    |x_1|\le\frac12,
    \qquad
    |\bm x_\perp|\le\frac{1}{\sqrt2\,n^3},
\]
we obtain
\[
\begin{aligned}
    \left|
    \kappa
    \left[
        \bigl(z_1(\bm u)-1\bigr)x_1
        +
        \bm u\cdot\bm x_\perp
    \right]
    \right|
    &\le
    \frac{\kappa|\bm u|^2}{2}
    +
    \frac{\kappa|\bm u|}{\sqrt2\,n^3}\le
    \frac{a_0^2}{2}
    +
    \frac{a_0\sqrt{\eps}}{\sqrt2}.
\end{aligned}
\]
Choose \(a_0>0\) and then \(\eps_0>0\) sufficiently small so
that the last expression is bounded by a fixed
\(\bx_0<\pi/2\). All summands then lie in the same fixed arc, and
hence, throughout the cap,
\[
    \left|
    \sum_{\bm x\in X_n}
    \alpha_{\bm x}
    e^{i\kappa\bm\bz\cdot\bm x}
    \right|
    \ge
    \cos(\bx_0)\sqrt{s_n}.
\]

Under the above parametrization,
\[
    d\bm\bz
    =
    \frac{d\bm u}{\sqrt{1-|\bm u|^2}}
    \ge d\bm u.
\]
Therefore the normalized surface measure of the cap is at least
\[
    \frac1{4\pi}
    \pi\frac{a_0^2}{\kappa}
    =
    \frac{a_0^2}{4\kappa}.
\]
It follows that
\[
\begin{aligned}
    \bm\alpha^*G_{X_n}\bm\alpha
    &=
    \frac1{4\pi}
    \int_{\mathbb S^2}
    \left|
    \sum_{\bm x\in X_n}
    \alpha_{\bm x}e^{i\kappa\bm\bz\cdot\bm x}
    \right|^2d\bm\bz\ge
    \cos^2(\bx_0)s_n
    \frac{a_0^2}{4\kappa}=
    \frac{a_0^2\cos^2(\bx_0)}{4\eps}.
\end{aligned}
\]
Thus, with
\[
    c_0:=\frac{a_0^2\cos^2(\bx_0)}4,
\]
we have
\[
    \lambda_{\max}(G_{X_n})
    \ge
    \bm\alpha^*G_{X_n}\bm\alpha
    \ge
    \frac{c_0}{\eps}.
\]
Finally,
\[
    \kappa\delta_{X_n}
    =
    \eps n^2
    =
    \eps s_n^{1/3},
\]
which completes the proof.
\end{proof}

\section{Sharpness of the absolute packing row sum}\label{app:row-sum-example}

Proposition~\ref{prop:absolute-row-sum-sharpness} proves that the
three-dimensional exponent \(2/3\) is optimal for the absolute
coherence row sum used in the Gershgorin argument.
\begin{proposition}
\label{prop:absolute-row-sum-sharpness}
There exists an absolute constant \(c>0\) such that, for arbitrarily
large \(s\), one can find a separated point cloud
\(X\subset B_1(0)\) and a wavenumber \(\kappa\) satisfying
\[
    \kappa\delta_X\asymp s^{2/3}
\]
and
\[
    \max_{\bm{x}\in X}
    \sum_{\substack{\bm{y}\in X\\\bm{y}\ne\bm{x}}}
    |K_\kappa(\bm{x}-\bm{y})|
    \ge c.
\]
More generally, the exponent \(2/3\) cannot be replaced by any
\(p<2/3\) in a uniform estimate of the form
\[
    \max_{\bm{x}\in X}
    \sum_{\substack{\bm{y}\in X\\\bm{y}\ne\bm{x}}}
    |K_\kappa(\bm{x}-\bm{y})|
    \lesssim
    \frac{s^p}{\kappa\delta_X}.
\]
\end{proposition}

\begin{proof}
For \(m\ge2\), set
\[
    \delta_m:=\frac{1}{\sqrt{3}\,m},
    \qquad
    X_m:=
    \left\{
        \delta_m\bm n:
        \bm n\in\mathbb Z^3,\quad |n_i|\le m
    \right\}.
\]
Then
\[
    X_m\subset \overline{ B_1(0)},
    \qquad
    \delta_{X_m}=\delta_m,
    \qquad
    s_m:=|X_m|=(2m+1)^3.
\]
Write
\[
    \rho:=\kappa\delta_m.
\]
The absolute row sum at the origin equals
\[
    S_m(\rho)
    :=
    \sum_{\substack{\bm n\in\mathbb Z^3\\
                    |n_i|\le m,\ \bm n\ne0}}
    \frac{|\sin(\rho|\bm n|)|}{\rho|\bm n|}.
\]

Let \(T\ge8\). The elementary bound
\[
    \int_a^{a+L}|\sin u|\,du
    \ge
    \frac{2L}{\pi}-2
\]
implies, uniformly for \(|\bm n|\ge1\),
\[
    \int_T^{2T}|\sin(\rho|\bm n|)|\,d\rho
    \gtrsim T.
\]
On the other hand, there are \(\gtrsim m^3\) lattice points in the
cube satisfying \(|\bm n|\asymp m\), and hence
\[
    \sum_{\substack{\bm n\in\mathbb Z^3\\
                    |n_i|\le m,\ \bm n\ne0}}
    \frac1{|\bm n|}
    \gtrsim m^2.
\]
Since \(\rho\le2T\) on the interval of integration,
\[
    \frac1T\int_T^{2T}S_m(\rho)\,d\rho
    \gtrsim
    \frac{m^2}{T}.
\]
Consequently, some \(\rho\in[T,2T]\) satisfies
\[
    S_m(\rho)
    \gtrsim
    \frac{m^2}{\rho}
    \asymp
    \frac{s_m^{2/3}}{\rho}.
\]

Taking \(T=m^2\) gives
\[
    \rho=\kappa\delta_m\asymp m^2\asymp s_m^{2/3}
\]
and
\[
    S_m(\rho)\gtrsim1.
\]
Finally, if a uniform estimate with an exponent \(p<2/3\) held, its
right-hand side along this sequence would be
\[
    \frac{s_m^p}{\kappa\delta_m}
    \asymp
    s_m^{p-2/3}
    \longrightarrow0,
\]
contradicting the lower bound above.
\end{proof}

\end{document}